\documentclass[11pt]{amsart}
\usepackage{geometry}                
\usepackage{graphicx}
\usepackage{amssymb}
\usepackage{epstopdf}
\newtheorem{teo}{Theorem}
\newtheorem{defi}[teo]{Definition}
\newtheorem{obs}[teo]{Remark}
\newtheorem{pro}[teo]{Proposition}
\newtheorem{cor}[teo]{Corollary}

\newcommand{\supp}{\mathrm{supp \,}}
\newcommand{\spa}{\mathrm{span \,}}
\newcommand{\ran}{\mathrm{ran \,}}
\newcommand{\mb}{\mathbb}
\newcommand{\mf}{\mathfrak}
\newcommand{\mc}{\mathcal}

\title{A Banach Space with a countable  infinite number of complex structures}
\author{\sc
W.  Cuellar Carrera }\thanks{This article is part of the  author's Ph.D  thesis   supervised by    Prof. V. Ferenczi  with support of FAPESP 2010/17512-6 \\ This paper was written during  a visit of the author at the ICMAT-Madrid supervised by Prof. J. Lopez-Abad and financed by FAPESP 2012/00631-8}
\address{Instituto de Matem\'atica e Estat\'istica Universidade de S\~ao Paulo\\ 	R. do Mat\~ao, 1010 - Butant\~a, S\~ao Paulo - Brazil}
\email{cuellar@ime.usp.br}

\begin{document}
\maketitle

\begin{abstract}
We give examples  of real Banach spaces with exactly infinite countably many complex structures and with $\omega_1$ many complex structures.
\end{abstract}
\section{Introduction} 
A real Banach   space  $X$ is said to admit a complex structure when there exists a linear operator  $I$ on $X$ such that $I^2=-Id$.  This turns $X$ into  a  $\mathbb C$-linear space by declaring a new law for the scalar multiplication:
\[  (\lambda + i\mu).x= \lambda x + \mu I(x)  \; \; \; (\lambda, \mu \in \mathbb R).
\]
Equipped with the equivalent norm
\[ \|x\|= \sup _{0\leq \theta\leq 2\pi} \| \cos \theta x+ \sin \theta Ix \|,
\]
we obtain a complex Banach space which will be denoted by $X^I$. The space $X^I$  is the complex structure of $X$ associated to the operator $I$, which is often referred itself as a complex structure for $X$. 

When the space $X$ is already a complex Banach space,  the operator $Ix=ix$  is a complex structure on $X_{\mathbb R}$ (i.e., $X$ seen as a real space) which generates $X$. Recall that for a complex Banach space $X$ its complex conjugate $\overline{X}$ is defined to be the space $X$ equipped with the new scalar multiplication $\lambda.x= \overline{\lambda}x$.

Two complex structures $I$ and $J$ on a real Banach space $X$ are equivalent if there exists a real automorphism $T$ on $X$  such that $TI=JT$. This is equivalent to saying  that the spaces $X^I$ and $X^J$ are $\mathbb C$-linearly isomorphic.  To see this, simply observe that  the relation $TI=JT$ actually means that  the operator $T$ is $\mathbb C$-linear as defined from $X^I$ to $X^J$.

We note that a complex structure $I$ on a real Banach space $X$ is an automorphism whose inverse is $-I$,  which is itself another complex structure on $X$. In fact, the   complex space $X^{-I}$ is   the complex conjugate space of $X^{I}$. Clearly  the spaces $X^{I}$ and $X^{-I}$ are always $\mathbb R$-linearly isometric. On the other hand, J. Bourgain \cite{B} and N. Kalton \cite{K} constructed examples of  complex  Banach spaces not  isomorphic to their corresponding complex conjugates, hence these spaces admit at least two different complex structures. Bourgain example is an $\ell_2$ sum of finite dimensional spaces whose distance to their conjugates tends to infinity. Kalton example is a twisted sum of two Hilbert spaces i.e., $X$  has a closed subspace $E$ such that $E$ and $X/E$ are Hilbertian, while $X$ itself is not isomorphic to a Hilbert space.

\

Complex structures do not always exist on Banach spaces. The first example in the literature  was the  James space,  proved by J. Dieudonn\'e \cite{D}. Other  examples of spaces without complex structures are the uniformly convex space constructed by S. Szarek  \cite{S}  and  the hereditary indecomposable space of  W. T. Gowers and B. Maurey  \cite{GM}. Gowers  \cite{G1,G2} also constructed a space  with unconditional basis but without complex structures. In general these spaces have few operators. For example, every operator on the Gowers-Maurey space  is a strictly singular perturbation of a multiple of the identity and this forbids complex structures:   suppose that $T$ is an operator on this space such that $T^2=-Id$ and write $T=\lambda Id+ S$ with $S$ a strictly singular operator. It follows that  $(\lambda^2+1)Id$ is  strictly singular  and of course this is impossible.

More examples of Banach spaces without complex structures were constructed by P. Koszmider, M. Mart\'in and J. Mer\'i \cite{KMM1, KMM2}. In fact, they introduced the notion of \emph{extremely non-complex Banach space}: A real Banach  space X is extremely non-complex if every bounded linear operator $T: X \rightarrow X$ satisfies the norm equality $ \|Id+ T^2\|= 1+ \|T\|^2$. Among their examples of extremely non complex spaces are  $C(K)$ spaces with few operators (e.g. when every bounded linear operator $T$ on $C(K)$ is of the form $T=gId+S$ where $g\in C(K)$ and $S$ is a weakly compact operator  on $C(K)$), a $C(K)$ space containing a complemented isomorphic copy of $\ell_{\infty}$ (thus having a richer space of operators than the first one mentioned) and an extremely non complex space  not isomorphic to any $C(K)$ space.  

\

Going back to the problem of uniqueness of complex structures,  Kalton proved that spaces whose complexification is a primary space have at most one complex structure \cite{FG}. In particular, the classical spaces $c_0$, $\ell_p$ $(1\leq p \leq \infty)$,  $L_p [0,1]$  $(1\leq p \leq \infty)$,  and $C[0,1]$ have a unique  complex structure.  

\

We have mentioned before examples of Banach spaces with at least two different complex structures. In fact, V. Ferenczi \cite{F} constructed  a space $X(\mb C)$ such that the complex structure $X(\mb C)^J$ associated to some operator $J$ and its conjugate are the only complex structures on $X(\mb C)$ up to isomorphism. Furthermore, every $\mb R$-linear operator $T$ on $X(\mb C)$ is of the form $T= \lambda Id + \mu J + S$, where $\lambda, \mu$ are reals and $S$ is strictly singular. Ferenczi also proved that the space $X(\mb C)^n$ has exactly $n+1$ complex structures for every positive integer $n$.   Going to the extreme,  R. Anisca \cite{A}  gave  examples of subspaces  of $L_p$ ($1\leq p <2$) which admit continuum many non-isomorphic complex structures. 

\

The question remains about finding examples of Banach spaces with exactly infinite countably many different complex structures. A first natural approach to solve this problem is  to construct an infinite sum of copies of $X(\mb C)$, and in  order to control the number of complex structures to take a regular sum,  for instance, $\ell_1(X(\mb C))$. It follows that every $\mb R$-linear bounded operator $T$ on $\ell_1(X(\mb C))$ is of the form $T= \lambda(T)+ S$, where $\lambda(T)$ is the scalar part of $T$, i.e., an infinite matrix of operators on $X(\mb C)$ of the form $\lambda_{i,j}Id+ \mu_{i,j}J$, and $S$ is an infinite matrix of strictly singular operators on $X(\mb C)$. It is easy to prove that if  $T$ is a complex structure then $\lambda(T)$ is also a complex structure. Recall  from \cite{F}  that two complex structures whose difference is strictly singular must be equivalent. Unfortunately, the operator $S$  in the representation of $T$ is not necessarily strictly singular, and this makes very difficult to understand the complex structures on $\ell_1(X(\mb C))$. 

\

It is necessary to consider  a more ``rigid" sum of copies of spaces like  $X(\mb C)$. We found this interesting property in the space $\mathfrak{X}_{\omega_1}$ constructed by S. Argyros, J. Lopez-Abad  and S. Todorcevic     \cite{ALT}.  Based on that construction we present a separable reflexive Banach space $\mathfrak{X}_{\omega^2}(\mathbb{C})$ with  exactly infinite countably many different  complex structures which  admits  an infinite dimensional Schauder decomposition $\mathfrak{X}_{\omega^2}(\mathbb{C})= \bigoplus_k \mathfrak{X}_{k}$ for which every $\mathbb R$-linear operator $T$ on $\mathfrak{X}_{\omega^2}(\mathbb{C})$ can be written as $T= D_T+ S$, where $S$ is strictly singular, $D_T |_{\mathfrak{X}_{k}}= \lambda_k Id_{\mathfrak{X}_{k}}$  $(\lambda_k \in \mathbb C)$ and  $(\lambda_k)_k$ is a convergent sequence. 

\

This construction also shows the existence of continuum many examples of  Banach spaces with the property of having exactly $\omega$ complex structures and the existence of a Banach space with exactly $\omega_{1}$ complex structures.

\section{ Construction of the space  $\mathfrak{X}_{\omega_1}(\mathbb{C})$ }

We construct a complex Banach space $\mathfrak{X}_{\omega_1}(\mathbb{C})$  with a bimonotone transfinite Schauder basis $(e_{\alpha})_{\alpha< \omega_{1}}$, such that every complex structure $I$ on $\mathfrak{X}_{\omega_1}(\mathbb{C})$ is of the form $I=D+S$, where $D$ is a suitable diagonal operator and $S$ is strictly singular. 

By a bimonotone transfinite Schauder basis we mean that $\mathfrak{X}_{\omega_1}(\mathbb{C}) = \overline{ \spa }(e_{\alpha})_{\alpha< \omega_{1}}$ and such that for every interval $I$ of $\omega_{1}$  the naturally defined map on the linear $\spa$ of $(e_{\alpha})_{\alpha<\omega_{1}}$
\[ \sum _{\alpha< \omega_{1}} \lambda_{\alpha} e_{\alpha}  \mapsto  \sum _{\alpha\in I} \lambda_{\alpha} e_{\alpha}
\]
extends to a bounded projection $P_{I}:  \mathfrak{X}_{\omega_1}(\mathbb{C}) \rightarrow \mathfrak X_{I}=  \overline{ \spa }_{\mb C}(e_{\alpha})_{\alpha\in I}$ with norm equal to 1.

Basically  $ \mathfrak{X}_{\omega_1}(\mathbb{C})$ corresponds to the complex version of the space $\mathfrak{X}_{\omega_1}$ constructed in \cite{ALT} modifying the construction in a way that its $\mb R$-linear operators have similar structural properties to the operators in the original space $\mathfrak{X}_{\omega_1}$ (i.e. the operators are strictly singular perturbation of a complex diagonal operator).

First we introduce the notation that will be used through all this paper. 

\subsection{ Basic notation}
Recall that $\omega$ and $\omega_1$ denotes the least infinite cardinal number and the least uncountable  cardinal number, respectively. Given ordinals $\gamma, \xi$  we write $\gamma+ \xi,\gamma \cdot \xi,\gamma ^{ \xi}$ for the usual arithmetic operations (see \cite{J}). For an ordinal $\gamma$ we denote  by $\Lambda(\gamma)$ the set of limit ordinals $<\gamma$.
 Denote by $c_{00} (\omega_1, \mathbb C)$ the vector space of  all functions $x: \omega_1 \rightarrow \mathbb C$ such that the set $\supp x= \{ \alpha< \omega_1 : x(\alpha)\neq 0\}$ is finite and by $(e_{\alpha})_{\alpha<\omega_1}$ its  canonical Hamel basis. For a vector $x\in c_{00} (\omega_1, \mathbb C)$   $\ran x$ will denote the minimal interval containing $\supp x$. Given two subsets $E_1$, $E_2$ of $\omega_1$ we say that $E_1<E_2$ if $\max E_1< \min E_2$. Then for $x,y\in c_{00} (\omega_1, \mathbb C)$   $x<y$   means that $\supp x < \supp y $.  For a vector $x\in c_{00} (\omega_1, \mathbb C)$  and a subset $E$ of $\omega_1$ we denote by $Ex$ (or  $P_Ex$) the restriction of $x$ on $E$  or simply the function $x\chi_E$. Finally in some cases we shall denote elements of $c_{00} (\omega_1, \mathbb C)$ as $f,g, h  \dots $ and its canonical  Hamel basis as $(e^{*}_{\alpha})_{\alpha<\omega_1}$ meaning that we refer to these elements as being functionals in the norming set. 

\subsection{ Definition of the norming set}

The space $\mf X_{\omega_1}(\mb C)$ shall be defined as the completion  of $c_{00} (\omega_1, \mathbb C)$ equipped with a norm given by a norming set $\mathcal{K}_{\omega_1}(\mathbb{C}) \subseteq c_{00} (\omega_1, \mathbb C)$. This means that the norm for every $x\in c_{00} (\omega_1, \mathbb C)$ is defined as  $\sup\{ |\phi(x)|= |\sum_{\alpha < \omega_1} \phi(\alpha) x(\alpha)|: \phi \in \mathcal{K}_{\omega_1}(\mathbb{C}) \}$. The norm of this space can also be defined inductively.  

\

We start by fixing two fast increasing sequences  $(m_j)$ and $(n_j)$  that are going to be used in the rest of this work. The sequences are defined recursively as follows:
\begin{itemize}

\item[1.] $m_{1}=2$ e $m_{j+1}=m^{4}_{j}$;
\item[2.] $n_{1}=4$ e $n_{j+1}=(4n_{j})^{s_j}$, where $s_j=\log_{2}m^{3}_{j+1}$.

\end{itemize}

\

Let  $\mathcal{K}_{\omega_1}(\mathbb{C})$  be  the minimal subset  of $c_{00}(\omega_1,\mathbb{C})$ such that 

\begin{itemize}
\item[1.] It contains every $e^{*}_{\alpha}, \,  \alpha < \omega_1$. It satisfies that  for every  $\phi \in \mathcal{K}_{\omega_1}(\mathbb{C})$ and for every complex number $\theta=\lambda+ i\mu$ with $\lambda$ and $\mu$ rationals and $|\theta|\leq 1$, $\theta \phi \in \mathcal{K}_{\omega_1}(\mathbb{C})$.  It is closed under restriction to intervals of $\omega_1$.

\item[2.] For every $\{ \phi_i,\, : \,   i=1,..., n_{2j}\}\subseteq \mathcal{K}_{\omega_1}(\mathbb{C})$ such that $ \phi_1<\cdots<    \phi_{n_{2j}}$, the combination
\begin{eqnarray*}
\phi=\displaystyle \frac{1}{m_{2j}}\sum^{n_{2j}}_{i=1}\phi_i\in \mathcal{K}_{\omega_1}(\mathbb{C}).
\end{eqnarray*}
In this case we say that $\phi$ is the result of an $(m^{-1}_{2j}, n_{2j})$-operation.

\item[3.] For every special sequence $(\phi_1, \ldots, \phi_{n_{2j+1}})$ (see the Definition \ref{tp1}), the combination
\begin{eqnarray*}
\phi=\displaystyle \frac{1}{m_{2j+1}}\sum^{n_{2j+1}}_{i=1}\phi_i\in \mathcal{K}_{\omega_1}(\mathbb{C}).
\end{eqnarray*}
In this case we say that $\phi$ is a special functional and that  $\phi$ is the result of an $(m^{-1}_{2j+1}, n_{2j+1})$-operation. 

\item[4.]  It is rationally convex.
\end{itemize}
Define a norm on $c_{00} (\omega_1, \mathbb C)$ by setting
\begin{eqnarray*}
 \|x\|  = \sup \left \{ \left | \sum_{\alpha < \omega_1} \phi(\alpha) x(\alpha)  \right | \, : \,  \phi \in \mathcal{K}_{\omega_1}(\mathbb{C})  \right\}.
\end{eqnarray*}
The space $\mathfrak{X}_{\omega_1}(\mathbb{C})$ is defined as the completion of   $(c_{00}(\omega_1,\mathbb{C}), \|.\|)$.

 \
 
This definition of the norming set $\mathcal{K}_{\omega_1}(\mathbb{C})$ is similar to the one in \cite{ALT}.  We add the property of being closed  under products with rational complex numbers of the unit ball. This, together with property 2 above, guarantees the existence of some type of sequences (like $\ell_{1}^{n}$-averages and $R.I.S$ see Appendix) in the same way they are constructed for $\mathfrak{X}_{\omega_1}$. It follows that the norm is also defined by 

\begin{eqnarray*}
 \|x\|  = \sup \left \{ \phi(x)=  \sum_{\alpha < \omega_1} \phi(\alpha) x(\alpha)  : \phi \in \mathcal{K}_{\omega_1}(\mathbb{C}),   \, \phi(x)\in \mathbb R \right\}.
\end{eqnarray*}
We also have the following  implicit formula for the norm:
$$
 \|x\| = \max  \left \{   \|x\|_{\infty},  \, \sup \sup_j \frac{1}{m_{2j}}\sum_{i=1}^{n_{2j}} \|E_ix\|,  \, E_1< E_2< \cdots <E_{n_{2j}}   \right \} \vee $$
 
\[  \sup \left \{  \frac{1}{m_{2j+1}}  \left | \sum_{i=1}^{n_{2j+1}}  \phi_i(Ex) \right | \, : \, (\phi_i)_{i=1}^{n_{2j+1}}  \,  \text{ is  $n_{2j+1}$- special,  $E$  interval} \right \}.
\]
It follows from the definition of the norming set that the canonical Hamel basis $(e_{\alpha})_{\alpha<\omega_{1}}$  is a transfinite bimonotone Schauder basis of $\mathfrak{X}_{\omega_1}(\mathbb{C})$. In fact, by Property  1 for every interval $I$ of $\omega_{1}$ the projection $P_{I}$ has norm 1:
\[ \| P_{I}x \|= \sup_{ f\in  \mathcal{K}_{\omega_1}(\mathbb{C})} | fP_{I}x|= \sup_{ f\in  \mathcal{K}_{\omega_1}(\mathbb{C})} | P_{I}fx|\leq \|x\|  
\]
 Moreover, we have that the basis $(e_{\alpha})_{\alpha< \omega_1}$ is boundedly complete and shrinking, the proof is the obvious modification to the one for $ \mathfrak{X}_{\omega_1}$ (see  \cite[Proposition 4.13] {ALT}). In consequence   $\mathfrak{X}_{\omega_1}(\mathbb{C})$ is reflexive. 

\begin{pro} \label{bal}$\overline{  \mathcal{K}_{\omega_1}(\mathbb{C}) }^{\omega^*} = \mathit B_{ \mathfrak{X}^*_{\omega_1}(\mathbb{C}) }$.
\end{pro}

\begin{proof} Recall that the set $\mathcal{K}_{\omega_1}(\mathbb{C})$ is by definition rational convex. We notice that $\overline{  \mathcal{K}_{\omega_1}(\mathbb{C}) }^{\omega^*}$ is actually a convex set.  Indeed let $f,g \in \overline{  \mathcal{K}_{\omega_1}(\mathbb{C}) }^{\omega^*}$ and $t\in (0,1)$. Suppose  that $f_{n} \overset{\omega^*}{\rightarrow} f$,  $g_{n} \overset{\omega^*}{\rightarrow} g$ and $t_{n} \rightarrow t$, where $f_{n}, g_{n} \in \mathcal{K}_{\omega_1}(\mathbb{C})$  and $t_{n} \in \mb Q\cap (0,1)$ for every $n\in \mb N$. Then $tf + (1-t)g \in \overline{  \mathcal{K}_{\omega_1}(\mathbb{C}) }^{\omega^*} $ because
\[ t_{n}f_{n} + (1-t_{n})g_{n} \overset{\omega^*}{\rightarrow}  tf + (1-t)g.
\]
In the same manner we can prove that $ \mathfrak{X}^*_{\omega_1}(\mathbb{C})$ is balanced i.e., $\lambda  \mathfrak{X}^*_{\omega_1}(\mathbb{C}) \subseteq  \mathfrak{X}^*_{\omega_1}(\mathbb{C})$ for every $|\lambda|\leq 1$.
To prove the Proposition suppose that there exists $f\in  \mathit B_{ \mathfrak{X}^*_{\omega_1}(\mathbb{C}) } \setminus \overline{  \mathcal{K}_{\omega_1}(\mathbb{C}) }^{\omega^*}$. It follows by a standard separation argument  that there exists $x\in \mathfrak{X}_{\omega_1}(\mathbb{C})$  such that
\[ |f(x)| > \sup\{ |g(x)|  : \,  g\in  \mathcal{K}_{\omega_1}(\mathbb{C})  \}
\]
which is  absurd.
\end{proof}
	
\section{Complex structures on  $\mathfrak{X}_{\omega_1}(\mathbb{C})$}

Let $I\subseteq \omega_{1}$ be an interval of ordinals, we denote by $\mathfrak{X}_{I}(\mathbb{C})$  the closed subspace of $\mathfrak{X}_{\omega_{1}}(\mathbb{C})$  generated by $ \{ e_{\alpha }\}_{\alpha \in  I}$. For every ordinal $\gamma<\omega_1$ we write   $\mathfrak{X}_{\gamma}(\mathbb C)= \mathfrak{X}_{ [0,\gamma) }(\mathbb C)$. Notice that $\mathfrak{X}_{I}(\mathbb{C})$ is a 1-complemented subspace  of $\mathfrak{X}_{\omega_{1}}(\mathbb{C})$: the restriction to coordinates in $I$ is a projection of norm 1 onto $\mathfrak{X}_{I}(\mathbb{C})$. We denote this projection by $P_I$ and by $P^{I}=(Id-P_{I})$ the corresponding  projection onto the complement  space  $(Id- P_I)\mathfrak{X}_{\omega_{1}}(\mathbb{C})$, which we denote  by $\mathfrak{X}^{I}(\mathbb{C})$.  

A  transfinite sequence $(y_{\alpha})_{\alpha<\gamma}$ is called a block  sequence when  $y_{\alpha}<y_{\beta}$ for all  $\alpha< \beta <\gamma$.  Given a block sequence $(y_{\alpha})_{\alpha<\gamma}$ a  \emph{block subsequence} of $(y_{\alpha})_{\alpha<\gamma}$ is a block sequence $(x_{\beta})_{\beta<\xi}$ in the span of $(y_{\alpha})_{\alpha<\gamma}$. A \emph{real block subsequence} of $(y_{\alpha})_{\alpha<\gamma}$ is a block subsequence in the \emph{real} span of $(y_{\alpha})_{\alpha<\gamma}$.  A sequence $(x_n)_{n\in \mb N}$ is a block sequence of $\mathfrak{X}_{\omega_{1}}(\mathbb{C})$ when it is a block subsequence of $(e_{\alpha})_{\alpha<\omega_1}$.

\begin{teo}\label{bt} Let $T:\mathfrak{X}_{\omega_1}(\mathbb{C})\to\mathfrak{X}_{\omega_1}(\mathbb{C})$ be a complex structure on $\mathfrak{X}_{\omega_1}(\mathbb{C})$, that is, $T$ is a bounded $\mathbb{R}$-linear operator such that $T^2=-Id$. Then there exists a bounded diagonal operator $D_T :\mathfrak{X}_{\omega_1}(\mathbb{C})\to\mathfrak{X}_{\omega_1}(\mathbb{C})$, which is another complex structure, such that $T-D_T$ is strictly singular. Moreover $D_{T}= \sum_{j=1}^{k}\epsilon_{j}iP_{I_{j}}$ for some signs $(\epsilon_{j})_{j=1}^{k}$  and  ordinal intervals $I_{1}< I_{2}< \ldots < I_{k}$ whose extremes are limit ordinals  and such that  $\omega_{1}=\cup_{j=1}^{k} I_{j}$.
\end{teo}
The strategy for the proof  of Theorem \ref{bt} is the same than the one in  \cite[Theorem 5.32]{ALT} for the real case. However here we want to understand bounded $\mb R$- linear operators in a complex space. This forces us to justify that the ideas from \cite{ALT} still work in our context. 
 The result is obtained using the  following theorems that we explain with more details in the Appendix.\\ \\
{\bf Step I.} There exists  a family  $\mathfrak F$ of semi normalized block subsequences of $(e_{\alpha})_{\alpha< \omega_{1}}$, called $R.I.S$ (\emph{Rapidly Increasing Sequences}), such that every normalized block sequence  $(x_n)_{n\in \mathbb N}$ of  $\mathfrak{X}_{\omega_1}(\mb C)$ has a real block subsequence in $\mathfrak F$.\\ \\
Recall that a Banach space $X$ is hereditarily indecomposable (or H.I) if no (closed) subspace of $X$ can be written as the direct sum of  infinite-dimensional subspaces. Equivalently, for any two subspaces $Y$, $Z$ of $X$ and $\epsilon>0$, there exist $y\in Y$, $z\in Z$ such that $\| y\|= \|z\|=1$ and $\|y-z\|<\epsilon$.\\ \\
{\bf Step II.} For every normalized block sequence $(x_n)_{n\in \mb N}$ of $\mf X_{\omega_{1}}(\mb C)$, the subspace $\overline{\spa}_{\mb R} (x_{n})_{n\in \mb N} $ of  $\mathfrak{X}_{\omega_1}(\mb C)$ is a real H.I space.\\ \\
{\bf Step III.} Let   $(x_{n})_{n\in \mathbb N}$ be a $R.I.S$ and $T:\overline {\spa}_{\mb C}(x_{n})_{n\in \mathbb N} \to \mathfrak{X}_{\omega_1}(\mathbb{C})$ be a bounded $\mathbb R$-linear operator. Then $\displaystyle\lim_{n \rightarrow \infty}d(Tx_n, \mathbb{C}x_n)=0$.  \\   \\ 
The proof of Step I, II and III are given in the Appendix.\\    \\
{\bf Step IV.} Let $(x_n)_{n\in \mathbb N}$ be a  $R.I.S$  and $T: \overline {\spa}_{\mb C}(x_{n})_{n\in \mathbb N} \to \mathfrak{X}_{\omega_1}(\mathbb{C})$ be  a bounded  $\mathbb R$-linear operator.  Then the sequence  $\lambda_{T}:\mathbb{N}\longrightarrow \mathbb{C}$ defined by $d(Tx_n, \mathbb{C}x_n)=\|Tx_n-\lambda_{T}(n)x_n\|$ is convergent.  

\begin{proof} [Proof of Step IV]
First we note that the sequence  $(\lambda_T(n))_n$ is  bounded. Then consider $(\alpha_n)_n$  and $(\beta_n)_n$  two strictly increasing sequences of positive integers  and suppose that $\lambda_{T}(\alpha_n)\longrightarrow \lambda_1$ and $\lambda_{T}(\beta_n)\longrightarrow \lambda_2$, when $n\longrightarrow \infty$. Going to a subsequence we can assume that  $x_{\alpha_n}<x_{\beta_n} < x_{\alpha _{n+1}}$ for every $n \in \mathbb{N}$.

\
Fix $\epsilon>0$. Using the result of the Step III,  we have that $\displaystyle\lim_{n \rightarrow \infty}\|Tx_{\alpha_n}-\lambda_1x_{\alpha_n}\|=0$.  By passing to a subsequence if necessary, assume 
$$\|Tx_{\alpha_{n}}-\lambda_{1}x_{\alpha_{n}}\|\leq \frac{\epsilon}{2^{n}6}.$$
for every $n\in \mathbb N$.  Hence, for every  $w= \sum_{n} a_nx_{\alpha_{n}}\in \spa_{\mathbb{R}}(x_{\alpha_{n}})_n$ with $\|w\|\leq 1$ we have

\begin{eqnarray*}
\|Tw-\lambda_1w\|&\leq& \displaystyle \sum_{n} |a_n|\|Tx_{\alpha_n}-\lambda_1x_{\alpha_n}\|\\
&\leq&\epsilon/3.
\end{eqnarray*}
because $(e_{\alpha})_{\alpha< \omega_1}$ is a bimonotone transfinite basis.
In the same way, we  can assume that   for every  $w \in \spa_{ \mathbb{R} }(x_{\beta_{m}})_m$ with $\|w\|\leq 1$, $\|Tw-\lambda_2w\|\leq\epsilon/3$.
By  Step II we have that  $\overline{ \spa}_{\mathbb R }(x_{\alpha_{n}})_n\cup (x_{\beta_{n}})_n$ is  real-H.I. Then there exist unit vectors $w_1\in \spa_{ \mathbb R }(x_{\alpha_{n}})_n$  and  $w_2 \in \spa_{\mathbb R}(x_{\beta_{m}})_m$, such that $\|w_1-w_2\|\leq \frac{\epsilon}{3}\|T\|$. Therefore,
\begin{eqnarray*}
\|\lambda_1w_1-\lambda_2w_2\| &\leq &\|Tw_1-\lambda_1w_1\|+\|Tw_1-Tw_2\|+\|Tw_2-\lambda_2w_2\| \leq \epsilon.
\end{eqnarray*}
By other side
\begin{eqnarray*}
\|\lambda_1w_1-\lambda_2w_2\|\geq \|(\lambda_1-\lambda_2)w_1\|-\|\lambda_2(w_1-w_2)\|= |\lambda_1-\lambda_2|-|\lambda_2|\epsilon.
\end{eqnarray*}
In consequence, $ |\lambda_1-\lambda_2|\leq (1+|\lambda_2|)\epsilon$. Since $\epsilon$ was arbitrary, it follows that $\lambda_1=\lambda_2$.
\end{proof}

\

Let  $T: \mathfrak{X}_{\omega_1}(\mathbb {C} ) \to \mathfrak{X}_{\omega_1}(\mathbb{C}) $ be  a bounded  $\mathbb R$-linear operator. There is a canonical way to associate a bounded diagonal operator $D_T$ (with respect to the basis $(e_{\gamma})_{\gamma<\omega_1}$) such that $T-D_T$ is strictly singular: Fix  $\alpha\in \Lambda(\omega_{1})$ a limit ordinal, and $(x_{n})_{n\in \mathbb N}$, $(y_{n})_{n\in \mathbb N}$  two $R.I.S$  such that $\sup_{n} \max \supp x_{n} = \sup_{n} \max \supp$ $y_{n}=\alpha+\omega$.  By a property of $\mf F$ we can mix the sequences $(x_{n})_{n}$ , $(y_{n})_{n}$  in order to form a new $R.I.S$ $(z_{n})_{n\in \mathbb N}$ such that $z_{2k} \in \{ x_{n}\}_{n\in \mathbb N}$  and $z_{2k-1} \in \{ y_{n}\}_{n\in \mathbb N}$ for all $k\in \mathbb N$ (See  Remark \ref{mr}).  Then it follows from  Step IV  that the sequences defined by the formulas  $d(Tx_n, \mathbb{C}x_n)=\|Tx_n-\lambda_{T}(n)x_n\|$ and $d(Ty_n, \mathbb{C}y_n)=\|Ty_n-\mu(n)y_n\|$ are convergent, and  by the mixing argument,  they must have the same limit. Hence for  each  $\alpha\in \Lambda(\omega_{1})$  there exists a unique complex number $\xi_{T}(\alpha)$ such that 
\[  \lim_{n\to \infty} \|Tw_{n} -\xi_{T}(\alpha)w_{n}\|=0
\]
for every $(w_{n})_{n\in \mb N}$  $R.I.S$ in $\mf X_{I_{\alpha}}$,  where  we write $I_{\alpha}$ to denote the ordinal interval $[\alpha, \alpha+\omega)$. We proceed to define a diagonal linear operator $D_{T} $  on the (linear) decomposition of $\spa (e_{\alpha})_{\alpha<\omega_{1}}$ 
\[ \spa (e_{\alpha})_{\alpha<\omega_{1}}= \bigoplus_{\alpha\in \Lambda(\omega_{1})} \spa (x_{\beta})_{\beta \in I_{\alpha}}
\]
by setting $D_{T}(e_{\beta})=\xi_{T}(\alpha)e_{\beta}$ when $\beta \in  I_{\alpha}$.
\

Observe in addition that this sequence $(\xi_{T}(\alpha))_{\alpha\in \Lambda(\omega_{1})}$ is convergent. That is, for every strictly increasing sequence $(\alpha_n)_{n\in \mathbb N}$ in $\Lambda(\omega_{1})$, the corresponding subsequence $(\xi_{T}(\alpha_n))_{n\in \mathbb N}$ is convergent. 
In fact,  for every $n\in \mathbb N$,  fix $(y_{n}^{k})_{k\in \mathbb N}$ a $R.I.S$ in $ \mathfrak{X}_{I_{\alpha_{n}}}$.  Then we can take $(y_{n}^{k_{n}})_{n\in \mb N}$  a $R.I.S$ such that $\| Ty_{n}^{k_{n}} - \xi_{T} (\alpha_{n}+ \omega)y_{n}^{k_{n}}\|< 1/n$. It follows  by  Step IV there exists $\lambda \in \mb C$ such that $\lim_{n}\| Ty_{n}^{k_{n}} - \lambda y_{n}^{k_{n}}\|=0$. This implies that $\lim_{n} \xi_{T}(\alpha_{n}+\omega)=\lambda$.

\

In general this operator $D_{T}$ defines a bounded operator on $\mf X_{\omega_{1}}(\mb C)$. The proof  is the same that in \cite[Proposition 5.31]{ALT} and uses that certain James like space of a mixed Tsirelson space is finitely interval representable in every normalized transfinite block sequence of $\mathfrak{X}_{\omega_1}(\mathbb{C})$. 
For  the case of complex structures we have a simpler proof (see Proposition \ref{dcs}).

\begin{pro}\label{opd} Let   $A$ be a subset of ordinals contained in $\omega_1$  and $X= \overline{\spa}_{\mb C} (e_{\alpha})_{\alpha\in A}$. Let $T:X \to \mf X_{\omega_{1}}(\mb C)$ be a bounded $\mb R$-linear operator. Then $T$  is strictly singular if and only if for every  $(y_{n})_{n\in \mb N}$  $R.I.S$ on $ X $, $\lim_{n} Ty_{n}=0$. 
\end{pro}

\begin{proof} The proposition is trivial when the set $A$ is finite, then we assume that $A$ is infinite. Suppose that $T$ is strictly singular. Let $(y_{n})_{n\in \mb N}$ be a   $R.I.S$ on $X$ such that $ \lim_{n} Ty_{n}\neq 0$,  then by Step IV  there is  $\lambda\neq 0$ with $ \lim_{n}\|Ty_{n} -\lambda y_{n}\|=0$. Take $0<\epsilon<|\lambda|$. By passing to a subsequence if necessary, we assume that $\|( T- \lambda Id)| _{\overline{\spa} (y_{n})_{n} } \|<\epsilon$. This implies that $T|_{\overline{\spa} (y_{n})_{n} }$ is an isomorphism which is a contradiction. 

Conversely, suppose that for every $(y_{n})_{n}$ $R.I.S$ on $X$, $\lim_{n} Ty_{n}=0$. Assume  that $T$ is not strictly singular. Then there is a block sequence subspace $Y= \overline{\spa} (y_{n})_{n\in \mathbb N}$  of $ X$ such that $T$ restricted to $Y$ is an isomorphism.  By  Step I we can  assume that the sequence $(y_{n})_{n}$ is already a $R.I.S$ on $X$.  Then $\inf_{n} \|Ty_{n}\|>0$. And we obtain a contradiction. 
\end{proof}

Given $Y \subseteq  \mf X_{\omega_{1}}(\mb C)$ we denote by $\iota_Y$ the canonical inclusion of $Y$ into $ \mf X_{\omega_{1}}(\mb C)$.
\begin{cor}\label{dia}Let  $\alpha \in \Lambda(\omega_{1})$ and $T: \mf X_{I_{\alpha}}(\mb C) \to \mf X_{\omega_{1}}(\mb C)$ be a bounded $\mb R$-linear operator. Then there exists (unique) $\xi_{T}(\alpha) \in \mb C$ such that $T - \xi_{T}(\alpha)\iota_{\mf X_{I_{\alpha}} (\mb C)}$ is strictly singular. 
\end{cor}
\begin{proof} Let $\xi_{T}(\alpha)$ be the (unique) complex number such that $\lim \| Ty_{n}- \xi_{T}(\alpha)y_{n}\|=0$ for every $(y_{n})_{n}$ $R.I.S$ on $\mf X_{I_{\alpha}}(\mb C)$. Then by the previous Proposition $T-\xi_{T}(\alpha)\iota_{\mf X_{I_{\alpha}} (\mb C)}$ is strictly singular.
\end{proof}

\begin{cor} Let $\alpha\in \Lambda(\omega_{1})$  and $R: \mf{X}_{I_{\alpha}}(\mb C) \to \mf{X}^{I_{\alpha}}(\mb C)$ be a bounded $\mb R$-linear operator. Then $R$ is strictly singular. 
\end{cor}

\begin{proof} By the previous result, $\iota_{\mf X^{I_{\alpha}}(\mb C)}R= \lambda_{\alpha}\iota_{\mf X_{I_{\alpha}} (\mb C)}+ S$ with $S$  strictly singular. Then projecting by $P^{I_{\alpha}}$ we obtain $R= P^{I_{\alpha}} \circ \iota_{\mf X^{I_{\alpha}}(\mb C)}R = P^{I_{\alpha}}S $ which is strictly singular.
\end{proof}

\begin{pro}\label{dcs} Let $T$ be a complex structure on $\mathfrak{X}_{\omega_1}(\mathbb{C})$. Then  the linear  operator $D_{T}$ is a bounded complex structure.
\end{pro}

\begin{proof} Let $T$ be a complex structure on $\mathfrak{X}_{\omega_1}(\mathbb{C})$ and $D_{T}$ the corresponding diagonal operator defined above. Fix  $\alpha\in \Lambda (\omega_1)$. We shall prove that $\xi_{T}(\alpha)^{2}=-1$.  In fact,
\begin{eqnarray*}
 T\circ \iota_{ \mf X_{I_{\alpha}}(\mb C)} &=&   P_{I_{\alpha}}T\circ \iota_{ \mf X_{I_{\alpha}}(\mb C)} +   P^{I_{\alpha}}T\circ \iota_{ \mf X_{I_{\alpha}}(\mb C)} \\
 &=&  P_{I_{\alpha}}T\circ \iota_{ \mf X_{I_{\alpha}}(\mb C)} + S_{1}  
\end{eqnarray*}
where $S_{1} $ is strictly singular. This implies $P_{I_{\alpha}}T\circ\iota_{ \mf X_{I_{\alpha}}(\mb C)}= \xi_{T}(\alpha)Id_{\mf X_{I_{\alpha}} (\mb C)} + S_{2}: \mf X_{I_{\alpha}}(\mb C) \to \mf X_{I_{\alpha}}(\mb C)$ with $S_{2}$ strictly singular.  Now computing:
\begin{eqnarray*}
 (P_{I_{\alpha}}T\iota_{ \mf X_{I_{\alpha}}(\mb C)} )\circ (P_{I_{\alpha}}T \iota_{ \mf X_{I_{\alpha}}(\mb C)} )&=&   P_{I_{\alpha}}T\circ P_{I_{\alpha}}T\iota_{ \mf X_{I_{\alpha}}(\mb C)} \\
 &=&  P_{I_{\alpha}}T\circ  (Id- P^{I_{\alpha}}) T\iota_{ \mf X_{I_{\alpha}}(\mb C)} \\
 &=& P_{I_{\alpha}}T^{2}\iota_{ \mf X_{I_{\alpha}}(\mb C)} - P_{I_{\alpha}}T \underline{P^{I_{\alpha}}T\iota_{ \mf X_{I_{\alpha}}(\mb C)} }\\
 &=& -Id_{\mf X_{I_{\alpha}}(\mb C) } + S_{3}
\end{eqnarray*}
where $S_{3}$ is strictly singular because the underlined operator is strictly singular.  Hence we have that $(\xi_{T}(\alpha)^{2} +1)Id_{ \mf X_{I_{\alpha}} }$ is strictly singular.Which allow us to conclude that $\xi_{T}(\alpha)^{2}=-1$. 
The continuity of $D_T$ is then guaranteed  by  the convergence of $(\xi_{T}(\alpha))_{\alpha\in \Lambda(\omega_1)}$. In deed, we have that there exist ordinal intervals $I_{1}< I_{2}< \ldots < I_{k}$ with $\omega_{1}=\cup_{j=1}^{k} I_{j}$ and such that $D_{T}= \sum_{j=1}^{k}\epsilon_{j}iP_{I_{j}}$ for some signs $(\epsilon_{j})_{j=1}^{n}$.

\end{proof}

\begin{obs} \label{hyp} More generally, the proof of Proposition \ref{dcs} actually shows that if $T$ is a $\mathbb R$-linear bounded operator on $\mathfrak{X}_{\omega_1}(\mathbb{C})$ such that $T^2+Id=S$ for some $S$ strictly singular, then $D_T$ is bounded and $D_T^2=-Id$.
\end{obs}

Now we can conclude the proof of Theorem \ref{bt}.
\begin{proof}[Proof of Theorem \ref{bt}] Let $T: \mathfrak{X}_{\omega_1}(\mathbb {C} ) \to \mathfrak{X}_{\omega_1}(\mathbb{C}) $ be  a bounded  $\mathbb R$-linear operator which is a complex structure and $D_{T}$ be the diagonal bounded operator  associated to it.  It only  remains to prove that $T-D_{T}$ is strictly singular.  And this follows directly from Proposition \ref{opd}, because by definition $\lim_{n} (T-D_{T})y_{n}=0$ for every  $(y_{n})_{n}$ $R.I.S$ on $\mathfrak{X}_{\omega_1}(\mathbb {C} )$.
\end{proof}

We come back to the study of the complex structures on $\mathfrak{X}_{\omega_{1}}(\mathbb{C})$. Denote by $\mathfrak{D}$ the family of complex structures $D_T$ on $\mathfrak{X}_{\omega_{1}}(\mathbb{C})$ as in Theorem \ref{bt}, i. e., $D_{T}=  \sum_{j=1}^{k}\epsilon_{j}iP_{I_{j}} $   where $(\epsilon_{j})_{j=1}^{k}$  are signs   and   $I_{1}< I_{2}< \ldots < I_{k}$ are ordinal intervals whose extremes are limit ordinals  and such that  $\omega_{1}=\cup_{j=1}^{k} I_{j}$. Notice that  $\mathfrak{D}$ has cardinality $\omega_{1}$.

\

Recall that two spaces are said to be incomparable if neither of them embed into the other. 

\begin{cor} The space  $\mathfrak{X}_{\omega_{1}}(\mathbb{C})$ has $\omega_{1}$ many complex structures up to isomorphism.  Moreover any two non-isomorphic complex structures are incomparable.
\end{cor}

\begin{proof} 
Let $J$ be a complex structure on $\mathfrak{X}_{\omega_{1}}(\mathbb{C})$. By Theorem \ref{bt} we have that $J$ is equivalent to one of the complex structures of the family $\mathfrak{D}$.

To complete the proof it is enough to show that given two different elements of $\mathfrak{D}$ they define non equivalent complex structures. Moreover, we prove that one structure does not embed into the other.  Fix $ J \neq K \in \mathfrak{D}$. Then  there exists an ordinal interval $I_{\alpha}=[\alpha,\alpha+\omega)$ such that,  without loss of generality,  $J|_{\mathfrak{X}_{I_{\alpha}}}= i Id|_{\mathfrak{X}_{I_{\alpha}}}$ and $K|_{\mathfrak{X}_{I_{\alpha}}}= -i Id|_{\mathfrak{X}_{I_{\alpha}}}$. Suppose that there exists $T: \mathfrak{X}_{\omega_{1}}(\mathbb{C})^{J} \rightarrow \mathfrak{X}_{\omega_{1}}(\mathbb{C})^{K}$ an isomorphic embedding. Then T is in particular a $\mb R$-linear operator such that $TJ = KT$.  We write using Corollary \ref{dia},  $T|_{\mathfrak{X}_{I_{\alpha}}}= \xi_{T}(\alpha)\iota_{\mf X_{I_{\alpha}} (\mb C)} +S$ with $S$ strictly singular. Then $\xi_{T}(\alpha) J|_{\mathfrak{X}_{I_{\alpha}}} -  \xi_{T}(\alpha) K|_{\mathfrak{X}_{I_{\alpha}}}=S_1$ where $S_1$ is strictly singular. Im particular for each $x\in \mathfrak{X}_{I_{\alpha}}$,  $S_1x=  2\xi_{T}(\alpha)i x$. It follows from the fact that $\mathfrak{X}_{I_{\alpha}}$ is infinite dimensional that $\xi_{T}(\alpha)=0$. Hence  $T|_{\mathfrak{X}_{I_{\alpha}}}= S$ but this a contradiction because $T$ is an isomorphic embedding.
\end{proof}

The next corollary offers uncountably many examples of Banach spaces with exactly countably many complex structures. 

\begin{cor}The space  $\mathfrak{X}_{\gamma}(\mathbb{C})$  has $\omega$ complex structures up to isomorphism  for every limit ordinal  $\omega^2 \leq \gamma < \omega_1$.
\end{cor}

\begin{proof} Let $J$ be a complex structure on $\mathfrak{X}_{\gamma}(\mathbb{C}) $. We extend $J$ to a complex structure defined in the whole space $\mathfrak{X}_{\omega_{1}}(\mathbb{C})$ by setting $T= JP_{I} + iP^{I}$, where $I=[0, \gamma)$.  It follows that $T= D_{T} +S$  for an strictly singular operator $S$ and a diagonal operator $D_{T}$ like in Theorem \ref{bt}. Notice that $D_{T}x=ix$ for every $x\in \mathfrak{X}^{I}$, otherwise there would be a limit ordinal $\alpha$ such that $S|_{\mathfrak X_{I_{\alpha}}}= 2iId|_{\mathfrak X_{I_{\alpha}}}$.  Hence $JP_{I}= D_{T}P_{I} + S$.  Which implies that  $J$ has the form $J= \sum_{j=1}^{k}\epsilon_{j}iP_{I_{j}} + S_{1}$   where $S_{1}$ is strictly singular on $\mathfrak{X}_{\omega_1}(\mathbb{C})$,  $(\epsilon_{j})_{j=1}^{k}$  are signs   and   $I_{1}< I_{2}< \ldots < I_{k}$ are ordinal intervals whose extremes are limit ordinals  and such that  $\gamma=\cup_{j=1}^{k} I_{j}$. Now the rest of the proof is identical to the proof of the previous corollary. In particular, all the non-isomorphic complex structures on $\mathfrak{X}_{\gamma}(\mathbb{C})$ are incomparable.
\end{proof}

We also have, using the same proof of the previous corollary, that for every increasing sequence of limit ordinals $A=(\alpha_n)_n$, the space $\mf X_A= \bigoplus_{n} \mf X_{I_{\alpha_n}}(\mb C)$, where $I_{\alpha_n}=[\alpha_n, \alpha_n+\omega)$, has exactly infinite countably many different complex structures. Hence there exists a family, with the cardinality of the continuum,  of Banach spaces   such that every  space in it has  exactly $\omega$ complex structures.  

\section{Question and Observations}

Is easy to check that subspaces of even codimension of a real Banach space with complex structure also admit complex structure. An interesting property of $\mathfrak{X}_{\omega_{1}}(\mathbb{C})$ is that any of its real hyperplanes (and thus every real subspace of odd codimension) do not admit complex structure.

\begin{pro} The real hyperplanes of $\mathfrak{X}_{\omega_{1}}(\mathbb{C})$ do not admit complex structure.
\end{pro}

\begin{proof}By the results of Ferenczi and E. Galego  \cite [Proposition 13]{FG2}  it is sufficient to prove that the ideal of all $\mb R$-linear strictly singular operators on $\mathfrak{X}_{\omega_{1}}(\mathbb{C})$ has the  lifting property, that is, for any $\mb R$-linear isomorphism on $\mathfrak{X}_{\omega_{1}}(\mathbb{C})$  such that $T^2+Id$ is strictly singular, there exists a strictly singular operator $S$ such that $(T-S)^2=-Id$. The proof now follows easily from the Remark \ref{hyp}.
\end{proof}

We now pass to present some open questions related to the results exposed in this paper. The first question is about a remark mentioned in the introduction and the Ferenczi's space $X(\mb C)$  with exactly two complex structures. 

\

{\bf Question 1.}
 For every $1\leq p <\infty$. How many complex structures has $\ell_p(X(\mb C))$?

\

Clearly the space $\mathfrak{X}_{\omega_{1}}(\mathbb{C})$  is non separable. Hence a natural question is:

\

{\bf Question 2.}
Does  there exist a separable Banach space with exactly $\omega_1$ complex structures?

\

{\bf Question 3.}
  Does there exist for every infinite cardinal $\kappa$  a Banach space with $\kappa$-many non-equivalent complex structures?
  
\

One open problem  in the theory of complex structure is to know if the existence of more regularity in the space guarantees that it admits unique complex structure.

\

{\bf Question 4.} 
Does there exist a real Banach space with unconditional basis admitting more than one complex structure?

\

The question is still interesting in spaces with even more regularity than an unconditional basis. For example, when  a real Banach space $X$ has a symmetric basis. In this case, $X$  admits at least one  complex structure, because it is isomorphic to its square. 

\

{\bf Question 5}
Does every real Banach space with symmetric basis have unique complex structure?

\

Question 5 is is strongly related with the well-known open problem:  \emph{Is every Banach space $X$, with a symmetric basis, primary?}.  In fact, a positive answer to this problem implies a positive solution for Question 5.  We just have to note that the complexification of a space with symmetric basis has  a symmetric basis, and recall  Kalton's result:  a Banach space such that its complexification is a primary space has unique complex structure.  

\section{Appendix}

The purpose of this section is to give a proof for the results in the Step I, II and III. Several proofs are very similar to the corresponding ones in \cite{ALT}.  In order to make this paper as self contained as possible, we reproduce them in detail.

\

First we clarify the definition of the norming set by defining what being a special sequence means. All the definitions we present in this part are the corresponding translation of \cite{ALT} for the complex case.

\subsection{Coding and Special sequences}

 Recall that $[\omega_1]^2= \{ (\alpha, \beta)\in \omega_1^2 \, : \, \alpha< \beta\}$.
 \begin{defi} A function $\varrho:[\omega_1]^2 \rightarrow \omega$ such that 
 
 \begin{enumerate}
 \item $\varrho (\alpha, \gamma) \leq \max\{\varrho(\alpha, \beta), \varrho(\beta, \gamma) \}$ for all $\alpha< \beta< \gamma<\omega_1$.
 
 \item $\varrho (\alpha, \beta) \leq \max\{\varrho(\alpha, \gamma), \varrho(\beta, \gamma) \}$ for all $\alpha< \beta< \gamma<\omega_1$.
 
 \item The set $\{ \alpha<\beta :  \varrho(\alpha,\beta)\leq n\}$ is finite for all $\beta<\omega_1$ and $n\in \mathbb N$
 
 \end{enumerate}
 is called a $\varrho$-function.
 \end{defi}
 The existence of $\varrho$-functions is due to Todorcevic  \cite{T}. Let us fix  a $\varrho$-function $\varrho:[\omega_1]^2 \rightarrow \omega$  and all the following  work relies on that particular choice of $\varrho$.
 
  \begin{defi} Let $F$  be a finite subset of $\omega_1$ and $p\in \mathbb N$ , we write 
 \[\rho_F=\rho_{\varrho}(F)= \max_{\alpha, \beta \in F}\varrho(\alpha, \beta).\] 
 \[ \overline{F}^p =\{ \alpha\leq \max F: \text{ there is $\beta \in F$ such that } \alpha\leq \beta  \text{ and }   \,\varrho(\alpha, \beta)\leq p\}
 \]
 \end{defi}
 
\textbf{$\sigma_{\varrho}$-coding and the special sequences}

We denote by $\mathbb Q_s(\omega_1, \mathbb C)$ the set of finite sequences $(\phi_1, w_1, p_1,  \ldots , \phi_d, w_d, p_d)$ such that 
  
  \begin{enumerate}
  \item For all  $i\leq d$,  $\phi_i\in c_{00}(\omega_1, \mathbb C)$ and  for all  $\alpha< \omega_1$  the real and the imaginary part of $\phi(\alpha)$ are rationals. 
  \item $(w_i)_{i=1}^d, (p_i)_{i=1}^d \in \mathbb N^d$ are strictly increasing sequences.
  \item $p_i\geq \rho_{( \cup_{k=1}^i \supp \phi_k )}$ for every $i\leq d$.
  \end{enumerate}
 
Let $\mathbb Q_s(\mathbb C)$ be the set of finite sequences $(\phi_1, w_1, p_1, \phi_2, w_2, p_2,  \ldots , \phi_d,$  $ w_d, p_d)$ satisfying properties  (1), (2) above and for every $i\leq d$, $\phi_i \in c_{00}(\omega, \mathbb C)$. Then $\mathbb Q_s(\mathbb C)$ is a countable set while $\mathbb Q_s(\omega_1, \mathbb C)$ has cardinality $\omega_1$. Fix a one to one function $\sigma:  \mathbb Q_s(\mathbb C) \rightarrow \{ 2j: j  \text{ is odd}\}$ such that 
 
 \[  \sigma (\phi_1, w_1, p_1, \ldots , \phi_d, w_d, p_d)> \max \{ p_d^2, \frac{1}{\epsilon^2}, \max \supp \phi_d\}
  \]
where $\epsilon = \min \{ | \phi_k(e_{\alpha})| : \alpha \in \supp \phi_k, \, k=1, \ldots, d\}$.  Given a finite subset $F$ of $\omega_1$, we denote by $\pi_F: \{ 1, 2, \ldots, \#F\} \rightarrow F$ the natural order preserving map, i.e.  $\pi_F$ is the increasing numeration of $F$.

\

Given $\Phi =(\phi_1, w_1, p_1, \ldots , \phi_d, w_d, p_d)\in \mathbb Q_s(\mathbb C)$, we set
  \[ G_{\Phi}= \overline{\cup_{i=1}^d \supp \phi_i}^{p_d}.
  \]
Consider the family $\pi_{G_{\Phi}}(\Phi)= (\pi_G(\phi_1), w_1, p_1, \pi_G(\phi_2), w_2, p_2, \ldots , \pi_G(\phi_d),$ $ w_d,  p_d)$ where
  
$$ \pi_G(\phi_k)(n)  = \left\{
\begin{array}{ll}
\displaystyle \phi_k(\pi_{G_{\Phi}} (n)), \quad \text{if} \, \, n\in G_{\Phi}  \\ \\
\displaystyle 0, \quad \text{otherwise.} \, \, 
\end{array}
\right.
$$
Finally $\sigma_p: \mathbb Q_s(\omega_1, \mathbb C) \rightarrow \{ 2j: \text{ $j$ odd}\}$ is defined by  $\sigma_p(\Phi)= \sigma(\pi_G(\Phi))$.

\begin{defi}\label{tp1} A sequence $\Phi= (\phi_1, \phi_2, \ldots, \phi_{n_{2j+1}})$ of functionals of  $ \mathcal{K}_{\omega_1}(\mathbb{C})$  is called  a $2j+1$ special sequence  if 

\

(SS 1.) $\supp \phi_1 < \supp \phi_2< \cdots< \supp \phi_{n_{2j+1}} $. For each $k\leq n_{2j+1}$, $\phi_k$ is of type $I$, $w(\phi_k)=m_{2j_k}$  with $j_1$ even and $m_{2j_1}> n_{2j+1}^2$.

\

(SS 2.) There exists a strictly increasing sequence $(p_1^{\Phi}, p_2^{\Phi}, \ldots, p_{n_{2j+1}-1}^{\Phi} )$  of naturals numbers such that for all $1\leq i\leq n_{2j+1}-1$ we have that $w(\phi_{i+1})= m_{\sigma_{\varrho}(\Phi_i)}$ where
\[  \Phi_i= (\phi_1, w(\phi_1), p_1^{\Phi}, \phi_2, w(\phi_2), p_2^{\Phi}, \ldots, \phi_i, w(\phi_i), p_i^{\Phi})
\]
\end{defi}

Special sequences in separable examples with  one to one codings are in general simpler: they are of the form $(\phi_{1}, w(\phi_{1}), \ldots , \phi_{k}, w(\phi_{k}))$. Their main feature  is that if  $(\phi_{1}, w(\phi_{1}), \ldots , \phi_{k}, w(\phi_{k}))$ and $(\psi_{1}, w(\psi_{1}), \ldots , \psi_{l}, w(\psi_{l}))$ are two of them, there exists  $i_{o}\leq \min \{ k,l\}$  with the property that 

\begin{eqnarray}\label{ss2} 
(\phi_{i}, w(\phi_{i}))= (\psi_{i}, w(\psi_{i})) \, \text{ for all} \, i\leq i_{0}\\
\{ w(\phi_{i}) \, : \, i_{0}\leq i \leq k\} \cap  \{ w(\psi_{i}) \, : \, i_{0}\leq i \leq l\}= \emptyset
\end{eqnarray}
In non-separable spaces, one to one codings are obviously impossible, and (1), (2) are no longer true. Fortunately, there is a similar feature to  (1), (2)  called the tree-like interference of a pair of special sequences (See \cite [Lemma 2.9]{ALT}):  Let $\Phi= (\phi_{1}, \ldots , \phi_{n_{2j+1}})$ and $\Psi= (\psi_{1}, \ldots , \psi_{n_{2j+1}})$ be two $2j+1$-special sequences, then there exist  two numbers $0\leq \kappa_{\Phi, \Psi} \leq \lambda_{\Phi, \Psi} \leq n_{2j+1}$  such that the following conditions hold:
\begin{itemize}
\item[TP. 1] For all $i\leq \lambda_{\Phi, \Psi}$,  $w(\phi_{i})= w(\psi_{i})$ and $p_{i}^{\Phi}= p_{i}^{\Psi}$.
\item[TP. 2] For all $i< \kappa_{\Phi, \Psi}$, $\phi_{i}=\psi_{i}$.
\item[TP. 3] For all $\kappa_{\Phi, \Psi}< i < \lambda_{\Phi, \Psi}$
\[ \supp \phi_{i} \cap \overline{ \supp \psi_{1} \cup \cdots \cup \supp \psi_{\lambda_{\Phi, \Psi} -1}}^{p_{\lambda_{\Phi, \Psi}}-1}= \emptyset \]\[
\text{and} \, \,  \supp \psi_{i} \cap \overline{ \supp \phi_{1} \cup \cdots \cup \supp \phi_{\lambda_{\Phi, \Psi}-1} }^{p_{\lambda_{\Phi, \Psi}}-1}= \emptyset
\]
\item[TP. 4] $\{ w(\phi_{i}) \, : \, \lambda_{\Phi, \Psi}< i \leq n_{2j+1} \} \cap \{ w(\psi_{i}) \, : \, i\leq n_{2j+1}\}=\emptyset$ and $\{ w(\psi_{i}) \, : \, \lambda_{\Phi, \Psi}< i \leq n_{2j+1}\} \cap \{ w(\phi_{i}) \, : \,  i\leq n_{2j+1}\}= \emptyset$.

\end{itemize}

 \subsection{Rapidly increasing sequences ($R.I.S$)}
 For the proof of Step I we shall construct a family of block sequences on $\mathfrak{X}_{\omega_{1}}(\mathbb{C})$ commonly called \emph{rapidly increasing sequences} ($R.I.S$). These sequences are very useful because one has good estimates of upper bounds on $|f(x)|$ for $f\in \mc K_{\omega_{1}}(\mb C)$ and $x$ averages of $R.I.S$.  
  
 For the construction of the family $\mf F$ the only difference from the general theory in \cite{ALT} is that our interest now is to study bounded $\mb R$-linear operators on the complex space $\mathfrak{X}_{\omega_{1}}(\mathbb{C})$. Hence, all the construction of $R.I.S$ in a particular block sequence $(x_n)_{n\in \mb N}$   must be on its \emph{real} linear span.  We point out here that there are no problems with this, because all the combinations of the vectors  $(x_{n})_{n\in \mb N}$ to obtain $R.I.S$ use rational scalars. 

\begin{defi}[$R.I.S$] We say that a block sequence  $(x_k)_k$ of $\mathfrak{X}_{\omega_1}(\mathbb{C})$ is a  $(C,\epsilon)-R.I.S$, $C, \epsilon>0$, when there exists a strictly increasing sequence of natural numbers $(j_k)_k$ such that:
\begin{itemize}
\item[(i)]$\|x_k\|\leq C$;
\item[(ii)]$| \mbox{supp} \; x_k|\leq m_{j_{k+1}}\epsilon$;
\item[(iii)] For all the functionals  $\phi$ of $\mathcal{K}_{\omega_1}(\mathbb{C})$ of type $I$,  with $\omega(\phi)<m_{j_k}$, $|\phi(x_k)|\leq \displaystyle \frac{C}{\omega(\phi)}$.
\end{itemize}
\end{defi}

The following remark is immediately consequence of  this definition.

\begin{obs}
Let  $\epsilon'< \epsilon$. Every $(C, \epsilon)$-R.I.S  has a subsequence which is a $(C, \epsilon')$-R.I.S.
 And for every strictly increasing sequence of ordinals $(\alpha_n)_n$ and every $\epsilon>0$, $(e_{\alpha_n})_n$ is a $(1,\epsilon)$-R.I.S.
\end{obs}

\begin{obs}\label{mr} Let $(x_{n})_{n}$ and $(y_{n})_{n}$ be two $(C,\epsilon)$- R.I.S such that $\sup_{n} \max $ $\supp  x_{n} = \sup_{n} \max \supp y_{n} $. Then there exists $(z_{n})_{n}$ a $(C,\epsilon)$- R.I.S. such that $z_{2n-1}\in \{ x_{k}\}_{k\in \mb N}$  and $z_{2n}\in \{y_{k}\}_{k\in \mb N}$.  
\end{obs}

\begin{proof} Suppose that $(t_{k})_{k}$ and $(s_{k})_{k}$ are  increasing sequences of  positive integers satisfying the definition of R.I.S for $(x_{k})_{k}$ and $(y_{k})_{k}$ respectively.  We construct $(z_{k})_{k}$  as follows. Let $z_{1}= x_{1}$ and $j_{1}= t_{1}$. Pick $s_{k_{1}}$ such that 
$ x_{1}< y_{s_{k_{1}}} $ and $t_{2}<s_{k_{1}}$. Then we define $j_{2}= s_{k_{1}}$ and $z_{2}= y_{s_{k_{1}}}$.  Notice that 

\begin{itemize}
\item[(i)]$\|z_1\|\leq C$;
\item[(ii)]$| \mbox{supp} \; z_1|\leq m_{t_{2}}\epsilon  \leq m_{s_{k_{1}}}\epsilon=m_{j_{2}}\epsilon$;
\item[(iii)] For all the functionals  $\phi$ of $\mathcal{K}_{\omega_1}(\mathbb{C})$ of type $I$,  with $\omega(\phi)<m_{j_1}$, $|\phi(z_1)|\leq \displaystyle \frac{C}{\omega(\phi)}$.
\end{itemize}
Continuing with this process we obtain the desired sequence. 

\end{proof}

\begin{teo}\label{w3} Let $(x_k)_k$ be a normalized block sequence of $\mathfrak{X}_{\omega_1}$  and  $\epsilon >0$. Then there exists a normalized block subsequence $(y_n)_n$ in $\spa_{\mathbb R} \{ x_k\}$  which is a $(3,\epsilon)-R.I.S$.
\end{teo}

For the proof of Theorem \ref{w3} we first construct a  simpler type of sequence.
\begin{defi} Let $X$ be a Banach space,  $C\geq 1$ and  $k \in \mathbb{N}$. A normalized vector  $y$  is called a $C-\ell^{k}_{1}$-average  of  $X$, when  there exist a block sequence $(x_1,...,x_k)$  such that
\begin{itemize}
\item[(1)] $y=(x_1+\ldots+x_k)/k$;
\item[(2)] $\|x_i\|\leq C$,  for all  $i=1,...,k$
\end{itemize}
\end{defi}

In the next result we want  to emphasize  that this special type of sequence are really constructed on the real structure of the space $\mf X_{\omega_{1}}(\mb C)$. 
\begin{teo}  \label{w1}For every normalized block sequence $(x_n)$ of  $\mathfrak{X}_{\omega_1}(\mathbb{C})$, and every integer $k$,  there exist $z_1<\ldots<z_k$ in  $\spa_{\mathbb{R}}(x_n)$, such that $\ (z_1+\ldots+z_k)/k$  is a  $2-\ell^{k}_{1}$-average.
\end{teo}

\begin{proof} 
The proof is standard. Suppose that the result is false. Let   $j$ and  $n$ be natural numbers with 
\begin{eqnarray*}
2^{n}>m_{2j}\\
n_{2j}>k^n.
\end{eqnarray*}

Let $N=k^n$ and  $x=\displaystyle \sum^{N}_{i=1}x_i$.  For each $1\leq i\leq n$ and every $1\leq j \leq k^{n-i}$, we define,
\begin{eqnarray*}
 x(i,j)=\displaystyle \sum^{jk^i}_{t=(j-1)k^i+1}x_t.
 \end{eqnarray*}
 
Hence, $x(0,j)=x_j$ and $x(n,1)=x$. \\
It is proved by induction on $i$ that $\|x(i,j)\|\leq 2^{-i}k^i$,  for all $i, j$.  
In  particular, $\|x\|=\|x(n,1)\|\leq 2^{-n}k^n= 2^{-n}N$.  Then by \emph{ Property 1.} of definition in the norming set
\begin{eqnarray*}
   \|x\|\geq \displaystyle \frac{1}{m_{2j}}\sum^{n_{2j}}_{t=1}\|x_t\|=\frac{n_{2j}}{m_{2j}}>\frac{N}{m_{2j}}. 
\end{eqnarray*}
Hence, 
\begin{eqnarray*}
  2^{-n}N&>&\frac{N}{m_{2j}}\\
  m_{2j}&>&2^n,
\end{eqnarray*}
which is is a contradiction.
\end{proof}

Finally,  for the construction of $R.I.S$  we observe these simple facts (\cite[Remark 4.10]{ALT})
\begin{itemize}
\item  If $y$ is a  $C-\ell^{n_j}_{1}$-average of  $\mathfrak{X}_{\omega_1}(\mathbb{C})$ and  $\phi \in \mathcal{K}_{\omega_1}(\mathbb{C})$  has weight  $\omega(\phi)<m_j$, then  $|\phi(y)|\leq \frac{3C}{2\omega(\phi)}$;
\item Let  $(x_k)_k$ be a block sequence of  $\mathfrak{X}_{\omega_1}(\mathbb{C})$ such that there exists a strictly increasing sequence of positive integers $(j_k)_k$ and $\epsilon>0$ satisfying:
\begin{itemize}
\item[a)] Each $x_k$ is a  $2-\ell^{n_{j_k}}_{1}$-average;
\item[b)] $|supp \; x_k|< \epsilon m_{j_{k+1}}$.

\end{itemize}
Then $(x_k)_k$ is a  $(3,\epsilon)-R.I.S$.
\end{itemize}

\subsection{Basic Inequality}

To prove Step II and III we need a  crucial result called  \emph{the basic inequality}  which is very important  to find  good estimations for  the norm of certain combinations of $R.I.S$  in  $\mf X_{\omega_{1}}(\mb C)$. First we need to introduce  the \emph{mixed Tsirelson spaces}.

The mixed Tsirelson space  $T [(m^{-1}_{j}, n_j)_j]$ is defined by  considering the completion of  $c_{00}(\omega, \mathbb{C})$ under  the norm  $\|. \|_{0}$  given by the following implicit formula
\begin{eqnarray*}
\|x\|_{0}=\max \left \{ \|x\|_{\infty}, \, \, \displaystyle \sup_{j} \sup \frac{1}{m_j}\sum^{n_j}_{i=1}\|E_ix\|_{0}\right \},
\end{eqnarray*}
The supremum  inside  the formula is taken over all the sequences  $E_1<\ldots<E_{n_j}$ of subsets of  $\omega$.   Notice that in this space the canonical Hamel basis $(e_{n})_{n<\omega}$ of $c_{00}(\omega, \mathbb{C})$ is  1-subsymmetric and 1-unconditional basis.

\

We can give an alternative definition for the norm of  $ T[(m^{-1}_{j}, n_j)_j]$  by defining the following norming set. Let  $W[(m_j^{-1}, n_j)] \subseteq c_{00}(\omega, \mathbb{C})$ the minimal set of $c_{00}(\omega, \mathbb{C})$  satisfying the following properties:
\begin{enumerate}
\item For every $\alpha<\omega$, $e^{*}_{\alpha}\in W[(m_j^{-1}, n_j)]$. If  $\phi \in W[(m_j^{-1}, n_j)] $ and $\theta=\lambda+ i\mu$ is a complex number with $\lambda$ and $\mu$ rationals and $|\theta|\leq 1$, $\theta \phi \in W[(m_j^{-1}, n_j)] $;
\item For every $\phi\in W[(m_j^{-1}, n_j)] $ and $E\subseteq \omega$, $E\phi \in W[(m_j^{-1}, n_j)] $;
\item For every $j\in \mathbb{N}$ and $\phi_1<\ldots<\phi_{n_j}$ in $W[(m_j^{-1}, n_j)] $, $(1/m_{j})\sum^{n_j}_{i=1}\phi_i\in W[(m_j^{-1}, n_j)] $;
\item $W[(m_j^{-1}, n_j)]$ is closed under convex rationals combinations. 
\end{enumerate}

 \begin{teo}[Basic Inequality for  $R.I.S$]  Let  $(x_n)_n$ be a  $(C,\epsilon)-R.I.S$ of  $\mathfrak{X}_{\omega_1}(\mathbb{C})$  and $(b_k)_k\in c_{00}(\mathbb{C},\mathbb{N})$.  Suppose that for some  $j_0 \in \mathbb{N}$ we have that for every $f \in \mathcal{K}_{\omega_1}(\mathbb{C})$  with weight  $w(f)=m_{j_0}$ and for every interval $E$ of $\omega_{1}$, 
 \begin{eqnarray*}
 \left|f\left(\sum_{k \in E}b_kx_k \right)\right|\leq C\left(\displaystyle \max_{k\in E}|b_k|+\epsilon\sum_{k \in E}|b_k |\right).
 \end{eqnarray*}
Then for every  $f\in \mathcal{K}_{\omega_1}(\mathbb{C})$ of type $I$,  there exist $g_1, g_2 \in c_{00}(\mathbb{C},\mathbb{N})$  such that
 \begin{eqnarray*}
 \left|f\left(\sum_{k \in E}b_kx_k \right)\right|\leq C(g_1+g_2)\left(\displaystyle \sum_{k \in E}|b_k |e_k\right),
 \end{eqnarray*} 
 where $g_1=h_1$ ou $g_1=e^{*}_{t}+h_1$, $t \notin supp\, h_1$, e $h_1 \in W[(m_j^{-1}, 4n_j)] $  such that $h_1 \in conv_{\mathbb{Q}}\left\{h\in W[(m_j^{-1}, 4n_j)]: w(f)=w(f)\right\}$ and $m_j$ does not appear as a weight  of a node in the tree analysis of $h_{1}$, and  $\|g_2\|_{\infty}\leq \epsilon$.
  \end{teo}
\begin{proof} See \cite [Section 8.2]{ALT}
\end{proof}
  
The following results are consequences of the basic inequality. The proof of this properties in our case is the same as in \cite{ALT}.
 \begin{pro} \label{cbi}Let  $f \in \mathcal{K}_{\omega_1}(\mathbb{C})$ or $f\in W[(m_j^{-1}, 4n_j)]$  of type  $I$. Consider $j \in \mathbb{N}$ and  $l\in \left[\displaystyle \frac{n_j}{m_j},n_j\right]$. Then for every set $F\subseteq c_{00}(\omega_{1}, \mb C)$  of cardinality $l$, 
 
$$\left|f\left(\frac{1}{l}\sum_{\alpha \in F}e_\alpha \right)\right|\leq \left\{
\begin{array}{ll}
\displaystyle \frac{2}{w(f)m_j}, \quad \text{if} \, \, w(f)<m_j,  \\ \\
\displaystyle \frac{1}{w(f)} \quad \text{if} \, \, w(f)\geq m_j.
\end{array}
\right.
$$
If the tree analysis of $f$ does not contain nodes of weight $m_j$, then
\[  \left|f\left(\frac{1}{l}\sum_{\alpha \in F}e_\alpha \right)\right|\leq \frac{2}{m_j^3}
\]
\end{pro} 
 
\begin{proof} \cite[Proposition 4.6]{ALT}
\end{proof}
 
\begin{pro} \label{db}Let  $(x_k)_k$ be a $(C,\epsilon)-R.I.S$ of  $\mathfrak{X}_{\omega_1}(\mathbb{C})$  with  $\epsilon \leq \displaystyle \frac{1}{n_j}$, $l\in \left[\displaystyle \frac{n_j}{m_j},n_j\right]$ and let  $f \in \mathcal{K}_{\omega_1}(\mathbb{C})$  of type $I$. Then, \\
 
$\left|f\left(\frac{1}{l}\sum^{l}_{k=1}x_k\right)\right|\leq \left\{
\begin{array}{ll}
\displaystyle \frac{3C}{w(f)m_j}, \quad \text{if} \, \, w(f)<m_j  \\ \\
\displaystyle \frac{C}{w(f)} + \frac{2C}{n_j}, \quad \text{if} \, \, w(f)\geq m_j.
\end{array}
\right.
$
\\
Consequentely,  if  $(x_k)^{l}_{k=1}$ is a  normalized $(C,\epsilon)-R.I.S$  with  $\epsilon\displaystyle  \leq \frac{1}{n_{2j}}$, $l \in \left[\displaystyle \frac{n_{2j}}{m_{2j}},n_{2j}\right]$,  then
\begin{eqnarray*}
\frac{1}{m_{2j}}\leq \left \Vert  \displaystyle \frac{1}{l}\sum^{l}_{k=1}x_k\right \Vert \leq \frac{2C}{m_{2j}}.
\end{eqnarray*}
\end{pro}

\begin{proof} 
Let  $(x_k)_k$ be a $(C,\epsilon)-R.I.S $ and take  $b=\left(\frac{1}{l},\ldots,\frac{1}{l},0,0,\ldots \right)\in c_{00}(\mathbb{N}, \mathbb{C})$. It  follows from the basic inequality  that for every  $f\in \mathcal{K}_{\omega_1}(\mathbb{C})$ of type $I$,  there exist $h_1\in W[(m_j^{-1}, 4n_j)]$  with $\omega(h_1)=\omega(f)$, $t\in \mathbb{N}$ and $g_2\in c_{00}(\mathbb{N})$ with $\|g\|_\infty\leq\epsilon$  such that
\begin{eqnarray*}
\left|f\left(\displaystyle\frac{1}{l}\sum^{l}_{k=1}x_k \right)\right|&\leq& C\left(e^{*}_{t}+h_1+g_2\right)\left(\frac{1}{l}\sum^{l}_{k=1} e_k \right).
\end{eqnarray*}
moreover, 
\begin{eqnarray*}
\left|  g_{2} \left(\displaystyle\frac{1}{l}\sum^{l}_{k=1}e_k \right)\right|&\leq& \|g_2\|_\infty \left\|\frac{1}{l}\sum_{k \in E}e_k\right\|_1 \leq \epsilon\leq \frac{1}{n_j}.
\end{eqnarray*} 

Now by  the estimatives on the auxiliary space $T[(m^{-1}_{j}, 4n_j)_j]$  of  the Proposition \ref{cbi},  we have

\begin{itemize}
\item If  $\omega(f)<m_j$,
\begin{eqnarray*}
\left|f\left(\displaystyle\frac{1}{l}\sum^{l}_{k=1}x_k \right)\right|&\leq& C\left(\frac{1}{l}+\frac{2}{\omega(f)m_j}+\frac{1}{n_j}\right)\\
&\leq& C\left(\frac{m_j}{n_j}+\frac{2}{\omega(f)m_j}+\frac{1}{n_j}\right)\\
&\leq& \frac{3C}{\omega(f)m_j}
\end{eqnarray*}
\item If  $\omega(f)\geq m_j$,
\begin{eqnarray*}
\left|f \left(\displaystyle\frac{1}{l}\sum^{l}_{k=1}x_k \right)\right|&\leq& C\left(\frac{1}{l}+\frac{C}{\omega(f)}+\frac{1}{n_j}\right)\\
&\leq& \frac{C}{\omega(f)}+\frac{2C}{n_j}
\end{eqnarray*}
\end{itemize}
And  notice 
\begin{itemize}
\item $\displaystyle \frac{3C}{\omega(f)m_{2j}}\leq \frac{2C}{m_{2j}}$,  if   $ \omega(f)<m_{2j}$,
\item $\displaystyle \frac{C}{\omega(f)}+\frac{2C}{n_{2j}}\leq \frac{C}{m_{2j}}+\frac{C}{m_{2j}}=\frac{2C}{m_{2j}}$, if  $\omega(f)\geq m_{2j}$.
\end{itemize}
 We conclude from the fact that  $\mathcal{K}_{\omega_1}(\mathbb{C})$  is the norming set: $$ \|  (1/l) \sum^{l}_{k=1} x_k  \|  \leq  2C /m_{2j}.$$
For the proof the second part of the theorem, let $(x_k)^{l}_{k=1}$  be a normalized $(C,\epsilon)-R.I.S$ with  $\epsilon\leq \frac{1}{n_{2j}}$, $l\in \left[\frac{n_{2j}}{m_{2j}},n_{2j}\right]$. For every  $k\leq l$,  we consider $x^{*}_{k}\in \mathcal{K}_{\omega_1}(\mathbb{C})$,  such that  $x^{*}_{k}(x_k)= 1$  and $\ran x^{*}_{k}\subseteq \ran x_k$, then  $x^*=\frac{1}{m_{2j}}\sum^{l}_{k=1}x^{*}_{k}\in \mathcal{K}_{\omega_1}(\mathbb{C})$ and  $x^*\left(\frac{1}{l}\sum^{l}_{k=1}x_k\right) = \frac{1}{m_{2j}}$. Hence, $\frac{1}{m_{2j}}\leq \left\| \frac{1}{l}\sum^{l}_{k=1}x_k\right\|$. 
\end{proof}

\subsection{Proof Step II}
Now we introduce another type of sequences in order to construct the conditional frame in $\mf X_{\omega_{1}}(\mb C)$. In fact, this space has no unconditional basic sequence.

\begin{defi} A pair $(x,\phi)$ with  $x \in \mathfrak{X}_{\omega_1}(\mathbb{C})$ and  $\phi \in \mathcal{K}_{\omega_1}(\mathbb{C})$, is called a $(C,j)$- exact pair when:
\begin{itemize}
\item[(a)] $\|x\|\leq C$,  $\omega(\phi)=m_j$ and  $\phi(x)=1$.
\item[(b)] For each $\psi \in  \mathcal{K}_{\omega_1}(\mathbb{C})$ of type $I$ and $\omega(x)=m_i$, $i\neq j$, we have

$ |\psi(x)|\leq \left\{
\begin{array}{ll}
\displaystyle \frac{2C}{m_i}, \quad \text{if} \, \, i<j  \\ \\
\displaystyle \frac{C}{m^{2}_{j}}  \quad \text{if} \, \, i>j
\end{array}
\right.
$
\end{itemize}
\end{defi}

\begin{pro} \label{ep} Let  $(x_n)_n$ be a normalized block sequence of  $\mathfrak{X}_{\omega_1}(\mathbb{C})$.  Then for every  $j \in \mathbb{N}$, there exist $(x,\phi)$  such that  $ x \in \spa_{\mathbb{R}}(x_n)$, $\phi \in \mathcal{K}_{\omega_1}(\mathbb{C})$ and  $(x,\phi)$ is a $(6,2j)$-exact pair.
\end{pro}

\begin{proof}  Fix $(x_n)_n$ a normalized block sequence of $\mathfrak{X}_{\omega_1}(\mathbb{C})$  and a positive integer $j$. By the Proposition \ref{w3}  there exists $(y_n)_n$ a normalized  $(3, 1/n_{2j})-R.I.S$    in  $\spa_{\mathbb{R}}(x_n)$.  For every  $1\leq i\leq n_{2j}$  and  $\epsilon>0$, we take $\phi_i \in \mathcal{K}_{\omega_1}(\mathbb{C})$ such that  $\phi_i(y_i)>1-\epsilon$ , and  $\phi_i< \phi_{i+1}$.  Let  $x= (m_{2j}/n_{2j} )\sum_{i=1}^{n_{2j}}y_i$ and  $\phi=(1/m_{2j})\sum_{i=1}^{n_{2j}}\phi_i \in  \mathcal{K}_{\omega_1}(\mathbb{C})$.  By perturbating $x$ by a rational coefficient on the support of some $y_i$ we may assume that then  $\phi(x)= 1$ and using Proposition \ref{db}  we conclude that $(x,\phi)$ is a $(6,2j)$-exact pair.
\end{proof}

\begin{defi} Let $j\in \mathbb N$. A sequence $(x_1, \phi_1,  \ldots, x_{n_{2j+1}}, \phi_{n_{2j+1}})$ is called  a $(1, j)$-dependent sequence when

\

(DS. 1) $\supp x_1 \cup \supp \phi_1 < \ldots < \supp x_{n_{2j+1}} \cup \supp \phi_{n_{2j+1}}$

\

(DS. 2) The sequence $\Phi = (\phi_1, \ldots , \phi_{n_{2j+1}})$ is a $2j+1$-special sequence.

\

(DS. 3) $(x_i, \phi_i)$ is a $(6, 2j_i)$-exact pair.  $ \#\supp x_i \leq m_{2_{j+1}}/n^2_{2j+1}$ for every $1\leq i \leq n_{2j+1}$ 

\ 

(DS. 4).  For every $(2j+1)$-special sequence $\Psi = (\psi_1, \ldots , \psi_{n_{2j+1}})$ we have that 
\[ \bigcup_{k_{\Phi,\Psi}<i<\lambda_{\Phi,\Psi}} \supp x_i \cap \bigcup_{k_{\Phi,\Psi}<i<\lambda_{\Phi,\Psi} }\supp \psi_i  = \emptyset.
\]
where $k_{\Phi,\Psi}, \lambda_{\Phi,\Psi}$ are numbers introduced in Definition \ref{tp1}.
\end{defi}

\begin{pro} For every   normalized block sequence   $(y_n)_n$ of $\mathfrak{X}_{\omega_1}(\mathbb{C})$, and every  natural number $j$ there exists a $(1,j)$-dependent sequence $(x_1, \phi_1, \ldots , $ $x_{n_{2j+1}},  \phi_{n_{2j+1}})$ such that $x_i $  is in the $\mb R$-span of  $(y_n)_n$ for every $i=1, \ldots , n_{2j+1}$.
\end{pro}

\begin{proof} 
Let  $(y_n)_n$ be a normalized block sequence of $\mathfrak{X}_{\omega_1}(\mathbb{C})$ and $j\in \mathbb N$.  We construct the sequence  $(x_1, \phi_1, \ldots , $ $x_{n_{2j+1}},  \phi_{n_{2j+1}})$ inductively. First using Proposition \ref{ep}  we choose   a $(6,2j_1)$-exact pair  $(x_1, \phi_1)$ such that $j_1$ is even,  $m_{2j_1}> n_{2j+1}^2$ and $x_1\in \spa_{\mathbb R} (y_n)_n$. Assume that we have constructed $(x_1, \phi_1, \ldots , $ $x_{l-1},  \phi_{l-1})$ such that there exists $(p_1, \ldots, p_{l-1})$ satisfying

\begin{enumerate}

\item  $\supp x_1 \cup \supp \phi_1 < \ldots < \supp x_{l-1} \cup \supp \phi_{l-1} $, where $x_i \in \spa_{\mathbb R} (y_n)_n$ and $(x_i, \phi_i)$ is a $(6,2j_i)$-exact pair.
\item For $1<i\leq l-1$, $w(\phi_i)= \sigma_{\varrho}(\phi_1, w(\phi_1), p_1,  \ldots, \phi_{i-1}, w(\phi_{i-1}), p_{i-1})$.
\item For $1\leq i<l-1$,  $p_i\geq \max\{ p_{i-1},p_{F_i} \}$, where $F_i = \cup_{k=1}^i \supp \phi_k \cup \supp x_k$. 
\end{enumerate}
To complete the inductive construction   choose  $p_{l-1}\geq \max \{ p_{l-2}, p_{F_{i-1}}, n_{2j+1}^2$ $\# \supp x_{l-1}\}$ and  $2j_l= \sigma_{\varrho} ( \phi_1, w(\phi_1), p_1,  \ldots, \phi_{l-1}, w(\phi_{l-1}), p_{l-1} )$.  Hence take a $(6,2j_l)$-exact pair $(x_l, \phi_l)$ such that $x_l \in  \spa_{\mathbb R} (y_n)_n$ and $\supp x_{l-1}\cup \supp \phi_{l-1} < \supp x_l \cup \supp \phi_l$. 
Notice that properties $(DS.1), (DS.2)$ and $(DS.3)$ are clear by definition of the sequence and $(DS.4)$ follows from (3) and $(TP.3)$.     
\end{proof}

Modifying a little the previous argument we obtain the following:

\begin{pro}\label{w6} For every  two  normalized block sequences $(y_n)_n$ and $(z_n)_n$  of $\mathfrak{X}_{\omega_1}(\mathbb{C})$, and every $j \in \mathbb N$ there exists a $(1,j)$-dependent sequence $(x_1, \phi_1, \ldots $ $\ldots, x_{n_{2j+1}}, \phi_{n_{2j+1}})$ such that $x_{2l-1} \in \spa_{\mathbb R}(y_n)$ and $x_{2l} \in \spa_{\mathbb R}(z_n)$ for every $l=1, \ldots , n_{2j+1}$.\qed
\end{pro}

Another consequence of the basic inequality is the following proposition.

\begin{pro}\label{w5} Let  $(x_1, \phi_1, ...,x_{n_{2j+1}}, \phi_{n_{2j+1}} )$ be a $(1,j)$dependent sequence. Then:
\begin{enumerate}
\item$\|\displaystyle\frac{1}{n_{2j+1}}\sum^{n_{2j+1}}_{i=1}x_i\|\geq\frac{1}{m_{2j+1}}$
\item$\|\displaystyle \frac{1}{n_{2j+1}}\sum^{n_{2j+1}}_{i=1}(-1)^{i+1}x_i\|\leq \frac{1}{m^{3}_{2j+1}}$
\end{enumerate} 
\end{pro}

\begin{proof}The first inequality is clear since the functional $\phi=1/m_{2j+1}\sum_{i=1}^{n_{2j+1}} \phi_i$ $ \in \mathcal{K}_{\omega_1}(\mathbb{C})$ and $\phi (\sum^{n_{2j+1}}_{i=1}x_i)= n_{2j+1}/m_{2j+1}$.   The second is obtained  by the Basic Inequality. For the complete proof see \cite[Proposition 3.7]{ALT}.
\end{proof}

We now can give a proof of Step II.

\begin{pro} \label{w8} Let $(y_n)_n$  be a normalized block sequence of  $\mathfrak{X}_{\omega_1}(\mathbb{C})$. Then  the closure of the \emph{real} span  of $(y_{n})_{n}$  is  H.I.
\end{pro}

\begin{proof}  Let $(y_n)_n$ be a  normalized block sequence  of $\mathfrak{X}_{\omega_1}(\mathbb{C})$.  Fix  $\epsilon> 0$  and two block subsequences $(z_n)_n$ and  $(w_n)_n$  in $\spa_{\mb R} (y_{n})_{n}$.  Take an integer  $j$ such that $m_{2j+1}\epsilon>1$. By Proposition  \ref{w6} there exist   a  $(1,j)$-dependent sequence  $(x_1, \phi_1, ...,x_{n_{2j+1}}, \phi_{n_{2j+1}} )$  such that  $x_{2i-1}\in \spa_{\mathbb{R}}(z_n)$ and  $x_{2i}\in \spa_{\mathbb{R}}(w_n)$. We define $z= (1/n_{2j+1})\sum^{n_{2j+1}}_{i=1(odd)}x_i$  and $w= 1/n_{2j+1}\sum^{n_{2j+1}}_{i=1(even)}x_i$. Notice that  $z \in \spa_{\mathbb{R}}(z_n)$ and  $w \in \spa_{\mathbb{R}}(w_n)$. Then by Proposition  \ref{w5}  we get  $\|z+w\|\geq 1/m_{2j+1}$  and  $\|z-w\|\geq 1/m^{2}_{2j+1}$. Hence $\|z-w\|\leq \epsilon \|z+w\|$.
\end{proof}

\subsection{ Proof of Step III}

\begin{defi}
A sequence $(z_1, \phi_1, \ldots , z_{n_{2j+1} }, \phi_{n_{2j+1} })$ is called a $(0,j)$- dependent sequence  when  it satisfies the following conditions:

\begin{itemize}
 \item (0DS.1) The sequence $\Phi= (\phi_1, \ldots, \phi_{n_{2j+1}})$ is a $2j+1$-special sequence and $\phi_i(z_k)=0$ for every $1\leq i,k \leq n_{2j+1}$.
 \item (0DS.2) There exists $\{ \psi_1, \ldots , \psi_{n_{2j+1}} \} \subseteq \mathcal{K}_{\omega_1}(\mathbb{C})$ such that $w(\psi_i)= w(\phi_i)$, $\# \supp z_i \leq w(\phi_{i+1})/n_{2j+1}^2$ and $(z_i, \psi_i)$ is a $(6, 2j_1)$- exact pair for every $1\leq i \leq n_{2j+1}$.
 \item (0DS.3) If $H=(h_1, \ldots , h_{n_{2j+1} })$ is an arbitrary $2j+1$-special sequence, then
 \[  \left ( \bigcup_{\kappa_{\Phi, H}<i< \lambda_{\Phi, H}} \supp z_i \right ) \cap  \left ( \bigcup_{\kappa_{\Phi, H}<i< \lambda_{\Phi, H}} \supp h_i \right ) = \emptyset.
  \]

\end{itemize}
\end{defi}

\begin{pro}  For every $(0,j)$-dependent sequence $(x_{1}, \phi_{1}, \ldots ,  x_{n_{2j+1}}, $ $\phi_{n_{2j+1}})$ we have that
\[ \left \|  \frac{1}{n_{2j+1}} \sum_{k=1}^{n_{2j+1}} x_{k} \right \|\leq \frac{1}{m_{2j+1}^{2}}.
\]
\end{pro}

\begin{proof} \cite[Proposition 5.23]{ALT}.
\end{proof}

\begin{pro} Let $(y_n)_n$ be a  $(C,\epsilon)-R.I.S$, $Y=\overline{\spa}_{\mb C}(y_n)$,  and  $T:Y\longrightarrow \mathfrak{X}_{\omega_1}(\mathbb{C})$  a $\mathbb{R}$-linear bounded operator. Then $\displaystyle\lim_{n \rightarrow \infty}d(Ty_n, \mathbb{C}y_n)=0$. 
\end{pro}

\begin{proof} 
Suppose that $\displaystyle\lim_{n \rightarrow \infty}d(Ty_n, \mathbb{C}y_n)\neq0$.  Then there exists  an infinite subset $B\subseteq \mathbb N$ such that $\inf_{n\in B} d(Ty_n, \mathbb{C}y_n)> 0$. We shall show that  for every $\epsilon>0$ there exists $y\in Y$ such that $\|y\|<\epsilon\|Ty\|$ and this is a contradiction.

{\bf Claim 1} There exists a limit ordinal $\gamma_0$, $A \subseteq \mathbb N$ infinite and $\delta>0$ such that 
\[ \inf_{n\in A} d(P_{\gamma_0}Ty_n, \mathbb C y_n)> \delta \]

To prove this claim we observe that

\[ \gamma_0= \min\{ \gamma<\omega_1 :  \exists A\in [\mathbb N]^{\infty}  \inf_{n\in A} d(P_{\gamma}Ty_n, \mathbb C y_n)> 0\}
\]
is a limit ordinal.  In fact,  by the assumption the set on the right side is not empty. And if $\gamma_0$ is not limit, then we have $\gamma_0= \beta +1$. The sequence $(y_n)_n$  is weakly null (because $(e_{\alpha})_{\alpha}$ is shrinking) and then
\[ \lim_{n\rightarrow \infty} e_{\beta+1}^*Ty_n=0
\]
And for large $n$ and every $\lambda \in \mathbb C$
\begin{eqnarray*}
 \|P_{\beta} Ty_n- \lambda y_n\| &\geq&   \|P_{\beta +1} Ty_n- \lambda y_n\|- \|e_{\beta+1}^*Ty_n \|\\
 &\geq& \delta-| e_{\beta+1}^*Ty_n|\geq \delta/2,
\end{eqnarray*}
which  is a contradiction.

\

{\bf Claim 2}  Fix $\gamma_0$ and $A\subseteq \mb N$ as in  Claim 1. Then there exist   a sequence $n_2< n_3< \dots $ in $A$, a sequence of functionals $f_2, f_3, \ldots$ in $\mc K_{\omega_1}(\mb C)$ and   a sequence of ordinals $\gamma_1< \gamma_2 < \ldots < \gamma_0$ such that

\begin{enumerate}
\item $d( P_{[\gamma_{k}, \gamma_{k+1}] }Ty_{n_k}, \mathbb C y_{n_k}) \geq \delta/2$
\item $f_kT_{y_{n_k}} \geq \delta/2$
\item $f_k(y_{n_k})=0$
\item $\ran f_k \subseteq \ran Ty_{n_k}$
\item $\supp f_k \cap \supp y_{n_m} = \emptyset$ when $m\neq k$
\end{enumerate}
To prove this claim, let $\xi= \sup \max \supp y_n$.  We analyze the three possibilities for $\xi$:

\

Case a.) $\xi < \gamma_0$. 
 Let $n_1= \min A$ and choose $\xi <\gamma_1 < \gamma_0$ such that 
\[ \| P_{\gamma_0}Ty_{n_1} - P_{\gamma_1}Ty_{n_1}\|< \delta/2,
\]
hence,  $d(P_{\gamma_1}Ty_{n_1}, \mathbb Cy_{n_1}) > \delta/2$. By minimality of $\gamma_0$ we have 
\[ \inf_{n\in A} d(P_{\gamma_1} Ty_n, \mathbb C y_n)=0,
\]
then we can choose $n_2>n_1$ in $A$ such that $d(P_{\gamma_1} Ty_{n_2}, \mathbb C y_{n_2})< \delta/2$ and this implies that
\[ d( (P_{\gamma_0}- P_{\gamma_1}) Ty_{n_2}, \mathbb C y_{n_2})> \delta/2.
\]
Approximating the vector $(P_{\gamma_0}- P_{\gamma_1}) Ty_{n_2}$  choose $\gamma_0> \gamma_2> \gamma_1$ such that
$ \| (P_{\gamma_0}- P_{\gamma_2}) Ty_{n_2}\| 
$ is so small in order to guarantee 
that  $$ d(P_{[\gamma_1, \gamma_2]} Ty_{n_2}, \mathbb C y_{n_2}) > \delta/2.$$ Using  the complex Hahn-Banach theorem, there exists $g_2\in \mathit B_{ \mathfrak{X}^*_{\omega_1}(\mathbb{C}) }$ such that 
\begin{itemize}
\item $ g_2(P_{[\gamma_1, \gamma_2]} Ty_{n_2})> \delta/2$
\item $g_2( y_{n_2})= 0$
\end{itemize}
and  by Proposition \ref{bal} we can choose  $h_2 \in  \mathcal{K}_{\omega_1}(\mathbb{C})$   such that
 $ h_2(P_{[\gamma_1, \gamma_2]} Ty_{n_2})> \delta/2$ and   $h_2( y_{n_2}) $ is arbitrarily small. Replacing $h_2$ by $\alpha h_2 + \beta k_2$ where $|\alpha|+ |\beta|=1$, $k_2(y_{n_2})$ is close enough to 1, and $k_2\in \mathcal{K}_{\omega_1}(\mathbb{C})$ we may assume that $h_2(y_{n_2})=0$.

Let $f_2= h_2 P_{[\gamma_1, \gamma_2] \cap \ran Ty_{n_2}} \in \mathcal{K}_{\omega_1}(\mathbb{C})$.  Again by minimality of $\gamma_0$, there exists $n_3>n_2$ in $A$ such that $d(P_{\gamma_{2}}Ty_{n_{3}}, \mb Cy_{n_{3}})< \delta/2$ and we can choose $\gamma_{0}>\gamma_{3}>\gamma_{2}$ satisfying
\[ d(P_{[\gamma_{2}, \gamma_{3}]}Ty_{n_{3}}, \mb C y_{n_{3}} ) > \delta/2.
\]
Again by Hahn-Banach  and by Proposition \ref{bal} there exists a functional $h_{3} \in \mc K_{\omega_{1}}(\mb C)$ such that
\begin{itemize}
\item $ h_3(P_{[\gamma_2, \gamma_3]} Ty_{n_3})> \delta/2$
\item $h_3( y_{n_3})=0$
\end{itemize}
then we define $f_3= h_3 P_{[\gamma_2, \gamma_3] \cap \ran Ty_{n_3}} \in \mathcal{K}_{\omega_1}(\mathbb{C})$.   The previous argument gives us the way to construct the sequences of  Claim 2.  Properties  (1)-(5) are easy to check, in particular property (5) is true because $\min \supp f_{k} > \xi> \max \supp y_{n_{l}}$ for every positive integers $k,l$.

\

Case b.) $\xi> \gamma_{0}$. In this case we start by picking $n_{1}\in A$ such that $\min \supp y_{n_{1}}> \gamma_{0}$. Then we repeat exactly the same argument that in case a.).

\

Case c.) $\xi=\gamma_{0}$.  We basically repeat the same argument of the case a.)  with the additional  care of maintaining property (5) true. That is, each time we choose the ordinal  $\gamma_{k+1}$ (with $\gamma_{0}> \gamma_{k+1}> \gamma_{k}$) we take it such that $\gamma_{k+1}>  \max \supp y_{n_{k+1}}$. 

\

{\bf Claim 3} There exists  a $(0,j)$- dependent sequence $(z_1, \phi_1, \ldots , z_{n_{2j+1}})$ such that 
\begin{itemize}
\item $z_i \in X$ for every $1\leq i \leq n_{2j+1}$
\item $\ran  \, \phi_k \subseteq \ran Ty_k$ and $\phi_k(Tz_k)> \delta/2$
\end{itemize}

Let $j$ with $m_{2j+1}> 24/\epsilon \delta$. Choose $j_1$ even such that $m_{2j_1} > n_{2j+1}^2$ (see definition of special sequence)  and $F_1 \subseteq A$ with $\#F_1= n_{2j_1}$ such that $(y_{n_k})_{k\in F_1}$ is a $(3, 1/n_{2j_1}^2)$-R.I.S.  Then define 

\[  \phi_1= \frac{1}{m_{2j_1}} \sum_{i\in F_1} f_i \in \mathcal{K}_{\omega_1}(\mathbb{C})  \, \text{  and  }   z_1= \frac{m_{2j_1}}{n_{2j_1}} \sum_{k\in F_1} y_k
\]
observe that $w(\phi_1)= m_{2j_1}$,  $\phi_1(Tz_1) = \frac{1}{n_{2j_1} } \sum_{i\in F_1} f_i \left (  \sum_{k\in F_1} Ty_k \right )> \delta/2$ and $\phi_1(z_1)= \frac{1}{n_{2j_1}} \sum_{i\in F_1} f_i( \sum_{k\in F_1}) =0$.  Select
\[ p_1 \geq \max \{ p_{\varrho}( \supp z_1 \cup \supp Tz_1 \cup \supp \phi_1),  n_{2j+1}^2\#\supp z_1\}
\]
denote $2j_2= \sigma_{\varrho}(\phi_1, m_{2j_1}, p_1)$. Then take $F_2 \subseteq A$ with $\# F_2= n_{2j_2}$ and $F_2 > F_1$ such that $(y_k)_{k\in F_2}$ is $(3, 1/n_{2j_2}^2)- R.I.S$ and define 

\[  \phi_2= \frac{1}{m_{2j_2}} \sum_{i\in F_2} f_i \in \mathcal{K}_{\omega_1}(\mathbb{C})  \, \text{  and  }   z_2= \frac{m_{2j_2}}{n_{2j_2}} \sum_{k\in F_2} y_k
\]
So we have $\phi_1< \phi_2$, $\phi_2(Tz_2)> \delta$ and $\phi_2(z_1)= \phi_2(z_2)=0$. Pick
\begin{eqnarray*}
 &p_2& \geq \max \{ p_1, p_{\varrho}( \supp z_1\cup \supp z_2 \cup \supp Tz_1 \cup \supp Tz_2 \cup \supp \phi_1  \cup  \\
 & & \supp \phi_2), n_{2j+1}^2\#\supp z_2 \}
   \end{eqnarray*}
and set $2j_3= \sigma_{\varrho}( \phi_1, m_{2j_1}, p_1, \phi_2, m_{2j_2}, p_2)$. Continuing  with this procedure  we form a sequence $(z_1, \phi_1, \ldots , z_{n_{2j+1} }, \phi_{n_{2j+1} })$. Now we  check that  this is a $(0,j)$-dependent sequence.

\

Property  (0DS.1) is clear, because of the construction of the functionals  their weights satisfies $w(\phi_{i+1})= m_{\sigma_{\varrho}(\Phi_i)}$ where $\Phi_i=  (\phi_1, w(\phi_1), p_1, \ldots$ ,  $\phi_i, w(\phi_i), p_i)$. 

Property (0DS.2)  We proceed to the construction of the sequence $\{ \psi_{1}, \dots ,$  $ \psi_{n_{2j+1} } \}$ in $ \mathcal{K}_{\omega_1}(\mathbb{C})$ such that  $(z_{i}, \psi_{i})$ is a $(6, 2j_{i})$-exact pair  and  $w(\psi_{i})= w(\phi_{i})$ for every $1\leq i \leq n_{2j+1}$.  The other condition $\# \supp z_i \leq w(\phi_{i+1})/n_{2j+1}^2$ is already obtained by the construction of the weights. For each $z_{i}$ there exists a subset $F_{i}\subseteq A$ with $\#F_{i}= n_{{2j_{i}}}$ such that  $z_{i}= (m_{2j_{i}}/ n_{2j_{i}}) \sum_{k\in F_{i}} y_{n_{k}}$  where $(y_{n_{k}})_{k\in F_{i}}$ is a $(3, 1/n_{2j_{i}} ^{2} )$ R.I.S.  Now we follow the same arguments as in Proposition \ref{ep}. For every $k\in F_{i}$ we take $f_{n_{k}}\in \mc K_{\omega_{1}}(\mb C)$  such that $f_{n_{k}}(y_{n_{k}}) = 1$ and $f_{n_{k}}<f_{n_{k+1}}$. Then $\psi_{i} = (1/m_{2j_{i}})\sum_{k\in F_{i}}f_{n_{k}}\in  K_{\omega_{1}}(\mb C)$  and $(z_{i}, \phi_{i})$ is a   $(6, 2j_{i})$-exact pair. 

Property (0DS. 3) Let $H=(h_{1}, \ldots , h_{n_{2j+1}})$ be an arbitrary $2j+1$-special sequence.  We consider two cases: a) Suppose that $\max \supp z_{k} \leq \max \supp \phi_{k}$ for every $1\leq k\leq n_{2j+1}$. Then $\supp z_{k} \subseteq \supp \overline{\phi_{\lambda_{\Phi, H}-1}}^{p_{\lambda_{\Phi, H}-1}}$ for every $ \kappa_{\Phi, H}< k< \lambda_{\Phi, H}$. Then for the second part of (TP. 3) we obtain the desired result. (b) Suppose  that $\max \supp \phi_{k} \leq  \max \supp z_{k}$ for every $1\leq k\leq n_{2j+1}$. Then $\supp \phi_{k} \subseteq \supp \overline{ z_{\lambda_{\Phi,H}-1}}^{p_{\lambda_{\Phi, H}-1}}$ for every  $ \kappa_{\Phi, H}< k< \lambda_{\Phi, H}$, and the result follows from the first part of (TP3).

\

Fix a $(0,j)$-dependent sequence as obtained in the previous claim, and define
\[ z=(1/n_{2j+1}) \sum_{k=1}^{n_{2j+1}} z_{k}  \,   \text{ and } \, \phi= (1/m_{2j+1}) \sum_{k=1}^{n_{2j+1}}\phi_{k}.
\]

Then $\phi (Tz)= (1/n_{2j+1}) \sum_{k=1}^{n_{2j+1}} \phi_{k}(Tz)\geq \delta/m_{2j+1}$ and $\|z\|\leq 12/m_{2j+1}^{2}$. Hence, $\| Tz\|\geq \delta/m_{2j+1}\geq \delta m_{2j+1} \|z\|/12> \epsilon \|z\|$ and this completes the proof.

\end{proof}

\bibliographystyle{unsrt}

\end{document}